\documentclass[11pt]{article}
\usepackage{latexsym}
\usepackage{amssymb}
\usepackage{graphicx}
\usepackage[T1]{fontenc}
\usepackage[utf8]{inputenc}
\usepackage{enumerate}
\usepackage{amssymb}
\usepackage{amsmath}
\usepackage{color}
\usepackage{tikz}
\usepackage{hyperref}
\hypersetup{hypertex=true,
	colorlinks=true,
	linkcolor=blue,
	anchorcolor=blue,
	citecolor=blue}
\usepackage{amsthm}
\usetikzlibrary{patterns,calc,intersections}
\usepackage{cleveref}
\crefformat{section}{\S#2#1#3} 
\crefformat{subsection}{\S#2#1#3} 
\crefformat{subsubsection}{\S#2#1#3}
\newtheorem{thm}{Theorem}[section]
\newtheorem{lem}[thm]{Lemma}
\newtheorem{cor}[thm]{Corollary}
\newtheorem{prop}[thm]{Proposition}
\newtheorem{rem}[thm]{Remark}

\newtheorem{defn}[thm]{Definition}

\newcommand{\be}{\begin{equation}}
\newcommand{\ee}{\end{equation}}
\numberwithin{equation}{section}
\def\p{{\partial}}

\def\N{\mathcal{N}}

\def\Hess{{\mathrm{Hess}}}
\def\a{{\alpha}}
\def\supp{\mathrm{supp}}
\def\cc{{\mathcal{C}}}

\def\dvol{\mathrm{dvol}}

\def\mR{{\mathcal{R}}}
\def\mS{{\mathcal{S}}}
\def\tH{{\tilde{H}}}

\newcommand{\RNum}[1]{\uppercase\expandafter{\romannumeral #1\relax}}

\def\l{\lambda}
\def\b{\beta}
\def\R{\mathbb{R}}

\def\Lip{{\mathrm{Lip}}}

\def\half{\frac{1}{2}}
\def\n{\frac{n}{2}}

\def\ep{\epsilon}

\newcommand{\ba}{\begin{aligned}}
	\newcommand{\ea}{\end{aligned}}

\def\Dom{\mathrm{Dom}}

\def\tvl{{\tilde{V}^\l}}
\begin{document}
	\def\o{\omega}

	\def\O{\Omega}

	\title{ The classical Weyl law for Schrödinger operators on complete Riemannian manifolds}
	\author{Maxim Braverman\footnote{Department of Mathematics, Northeastern University, Boston, MA, 02215, m.braverman@northeastern.edu} 
    \and Xianzhe Dai\footnote{Department of Mathematics, UCSB, Santa Barbara, CA 93106, dai@math.ucsb.edu. } 
    \and Junrong Yan\footnote{Department of Mathematics, Northeastern University, Boston, MA, 02215, j.yan@northeastern.edu}
	}
	\date{}                                           
	\maketitle
	\abstract{ Building on and extending the analytical framework established in \cite{WeylDY}, we establish a criterion for the validity of the classical (non-semiclassical) Weyl law for
Schrödinger operators
$
H=\Delta+V
$
on complete Riemannian manifolds. 
In contrast to existing results, our approach does not rely on standard geometric assumptions such as bounded geometry, nor on analytic assumptions such as the doubling condition on the potential.  
Instead, we identify a geometric–analytic invariant that encodes the precise balance
between the geometry of the manifold, the growth of $V$, and the oscillation scale of~$V$.
This intrinsic quantity, denoted 
$c_{\delta}(\l)$ admits effective quantitative estimates. We prove that the Weyl asymptotic
\eqref{classical Weyl asymptotic} holds provided
$
\lim_{\lambda\to\infty} c_\delta(\lambda)=0 .
$
The sharpness of this criterion is demonstrated through 
explicit examples showing that the Weyl law can fail when the criterion is 
violated.
}

		\section{Introduction}

\subsection{Overview}
The spectral asymptotics of Schr\"odinger operators
$
H=\Delta+V
$
on $n$ dimensional complete Riemannian manifolds $(M,g)$ (c.f.   \eqref{classical Weyl asymptotic}) remain an active subject of study. A fundamental problem, going back to Weyl \cite{weyl1911asymptotische}, is to determine when the eigenvalue
counting function  $\N(\l)$ of $H$ satisfies the classical asymptotic
\be\label{classical Weyl asymptotic}
\N(\lambda) \sim (2\pi)^{-n} \int_{T^*M} \mathbf{1}_{\{|\xi|_g^2 + V(x) \le \lambda\}}\, \dvol_{T^*M}(x,\xi)= (2 \pi)^{-n} \omega_{n}\int_M(\l-V(x))_+^{\n} \dvol_M(x),
\ee
where $\omega_n$ is the volume of unit Euclidean ball in $\R^n$.

In compact settings this formula is well known, but on noncompact manifolds the
interaction between the geometry, the growth of~$V$, and its oscillations at
infinity creates a subtle picture, with many unresolved aspects; see \Cref{nontrivial} 
for illustrative examples.

Many existing Weyl laws on noncompact manifolds require strong geometric
assumptions, e.g. geometry at infinity admits an explicit
asymptotic structure, such as Euclidean or hyperbolic space,
asymptotically flat manifolds, cylindrical ends, or asymptotically hyperbolic
geometries \cite{de1950asymptotic, rozenbljum1974asymptotics, feigin1976asymptotic, fleckinger1981estimate,tachizawa1992eigenvalue,levendorskiui1996spectral,inahama2004eigenvalue,inahama2004eigenvalues,moroianu2008weyl, bonthonneau2015weyl, coriasco2021weyl, MP-WEYL, chitour2024weyl}. A recent theorem of Dai-Yan~\cite{WeylDY} establishes
\eqref{classical Weyl asymptotic} for general complete manifolds
with bounded geometry under a natural \emph{doubling condition} on the
potential $V$, namely, for some constant $C>1$,
\be\label{doubling condition0}
\text{the measure of } \{ V < 2\l \}
\le
C\cdot\,\text{the measure of } \{ V < \l \},
\ee
together with very mild regularity assumptions.
Within the doubling setting, these hypotheses are essentially optimal.
  Indeed, most existing Weyl-type results rely heavily on the doubling condition \eqref{doubling condition0} (or stronger ones), such as those in \cite{rozenbljum1974asymptotics,feigin1976asymptotic,fleckinger1981estimate,tachizawa1992eigenvalue,inahama2004eigenvalue,inahama2004eigenvalues,WeylDY}.

However, the doubling condition \eqref{doubling condition0} still rule out many geometrically significant
examples. Most notably, manifolds admitting potentials with very slow
growth that nevertheless satisfy Weyl’s law
\eqref{classical Weyl asymptotic} are excluded, such as the square of the
distance function on hyperbolic space, or $\ln(|x|)$ on $\R^n$, for which
the doubling condition \eqref{doubling condition0} fails
dramatically.
  Even when the full 
spectrum is discrete; no general principle has been available that explains
precisely when the classical Weyl law should remain valid.

This raises the natural question: is there a clean characterization of when the classical Weyl asymptotic \eqref{classical Weyl asymptotic} holds? The present paper is motivated by this question. Building on and extending the techniques established in \cite{WeylDY}, we develop a framework that goes beyond the doubling condition and identifies a geometric-analytic invariant that ensures the validity of Weyl's law.

\subsubsection{A new geometric-analytic invariant and its physical interpretation}

The first contribution of this paper is the identification of an accessible
geometric-analytic invariant
$
c_\delta(\lambda)
$ (see \eqref{c delta}),
which governs the validity of \eqref{classical Weyl asymptotic}.  
At the spectral scale~$\lambda$, this quantity captures the interaction between

\begin{itemize}
    \item the geometry of the manifold,
    \item the growth scale $a(\lambda)$ of $V$, and 
    \item the oscillation scale $b_\delta(\lambda)$ of $V$ on the region where $V(x)\approx\lambda$.
\end{itemize}
The condition $\lim_{\lambda\to\infty} c_\delta(\lambda)=0$ expresses the balance between these three ingredients, and our main theorem shows that this balance is exactly what ensures the Weyl law \eqref{classical Weyl asymptotic}.  
This yields a  nearly optimal criterion for Schr\"odinger operators on complete manifolds.

From a physics perspective, the invariant $c_\delta(\lambda)$ admits a clear
interpretation.
For the Weyl asymptotics to hold, local oscillations of the potential should not
dominate the large-scale geometric structure; otherwise, effects analogous to
quantum tunneling or diffraction may become strong enough to invalidate the
classical phase-space picture.
Traditional formulations of the Weyl law therefore impose global regularity
assumptions (such as compactness, doubling conditions, or large-scale uniformity)
precisely to rule out such phenomena.
The condition $c_\delta(\lambda)\to 0$ allows the manifold $(M,g)$ to exhibit
substantial geometric irregularities, provided that the combined growth and
oscillation of the potential $V$ compensates for them.
In this sense, a mild potential $V$ can ``tame'' a wild underlying geometry at high
energies, restoring classical behavior and ensuring the validity of the Weyl law.
Conversely, sufficiently mild geometry can also ``tame'' a wild potential.

\subsubsection{Nontriviality and sharpness of our criterion}\label{nontrivial}

The following examples show that deriving the Weyl asymptotic 
\eqref{classical Weyl asymptotic} without a doubling condition on the potential is  a delicate matter and demonstrate the sharpness of our criterion.  
Moreover, in contrast to the semiclassical Weyl asymptotics for semiclassical Schr\"odinger operator $H_{\hbar}:=\hbar^2\Delta+V$, which remain valid 
for the potentials considered below and can be obtained via localization 
methods~\cite{braverman2025semi}, the classical Weyl law is substantially more 
delicate.  These phenomena highlight the necessity of the new tools developed
here and in~\cite{WeylDY}.

On $\R$, let $c>0$ be a sufficiently small constant and $V$ be a smooth function satisfying
\[
V(x) =c\cdot (\ln |x|)^a \qquad (|x|\gg 1,\ a>0),
\quad\text{and}\quad
H^{\R} = -\p_x^2 + V(x).
\]
Let $S^{1}$ be the unit circle, $\Delta^{S^{1}}$ its Laplacian, and consider the
product operator on $\R \times S^{1}$:
\[
H := H^{\R} + \Delta^{S^{1}}.
\]
In \Cref{hyperbolic example}, we show that when $a \in (0,1]$, the potential $V$
does not satisfy our criterion.
This failure stems from the fact that the injectivity radius of $\R \times S^{1}$
is uniformly bounded above.
Moreover, we verify that
\[
\int_{S^{1}}\!\int_{\R} \big(\lambda - V(x)\big)_{+}\, dx\, d\theta
= o\big(\N(\lambda)\big),
\qquad \lambda \to \infty,
\]
and consequently the classical Weyl law fails for $H$. When $a \in (1,\infty)$, we have $\lim_{\l \to \infty} c_\delta(\l) = 0$ (see \cref{examples satisfy assumptions}), so our
theorem applies and the Weyl asymptotic \eqref{classical Weyl asymptotic} holds.

A second, more geometric example is as follows. Consider
\[
V(x)=r(x)^{a}, \qquad r(x)=d(x,0),
\]
on the real hyperbolic space. None of the existing approaches applies in this setting.\footnote{Although
\cite{inahama2004eigenvalue,inahama2004eigenvalues} establish Weyl-type asymptotics
on hyperbolic spaces under an assumption of the form
\[
\int_{\mathbb{H}^n} (\lambda - V(x))_{+}^{\n}\, dx \sim c\, \lambda^{\gamma},
\qquad \lambda \to \infty,
\]
this assumption does not apply to potentials such as $V(x) = r(x)^{\alpha}$.}
In this model, $c_\delta(\lambda)\to0$ holds precisely when $a>1$ (see \cref{examples satisfy assumptions}).  
We prove that the classical Weyl law holds exactly in this regime, whereas for 
$0<a<1$ not only does $\lim_{\l\to\infty}c_\delta(\lambda)$ fail to vanish, but the Weyl law 
itself breaks down (see \Cref{hyperbolic example}).  

		\subsection{Notations and the main result}\label{notation and main result}
Throughout this paper, the Riemannian manifold $(M,g)$ is complete
and  $$n=\dim(M).$$
		Given such a manifold $(M, g)$, let $\Delta$ denote the Laplace-Beltrami operator. (Our sign convention for the Laplace operator is the one that makes $\Delta$ a nonnegative operator.) The corresponding Schrödinger operator $H$ on $(M^n, g)$ takes the form $\Delta + V$, where $V \in C(M)$\footnote{While our methods can also be formulated for potentials with weaker regularity,
as in \cite{WeylDY} where an integral oscillation condition is used,
we focus here on the case of continuous potentials to keep the exposition
simple.

} is the potential function.
In this paper, we assume
\be\label{cond-V-2}
\lim_{d(x, x_0) \to \infty} V(x) = \infty,
\ee
where $d$ is the distance function induced by $g$ and $x_0$ is some fixed point. It is well known that under these conditions, the operator $H$ (which is $\Delta + V$) is essentially self-adjoint (cf. \cite{rofe1970conditions,oleinik1994connection}). Moreover, the spectrum of $H$ is discrete, and each eigenvalue has finite multiplicity.

Consider
\be\label{defn-Phi}
\Phi(\l):=(2 \pi)^{-n} \omega_{n}\int_{M}\big(\l-V(x)\big)_+^\n\,\dvol_M(x).
\ee
Without loss of generality, we assume $$V\geq1\quad\text{and}\quad\Phi(2)\neq 0.$$
\textbf{Quantitative description of the growth of $V$.}
We first introduce two quantities, $a(\l)$ and $d_\delta(\l)$, to describe
the growth of $V$. The quantity $d_\delta(\l)$ also depends on $\delta>0$; later on, we will let $\delta \to 0$.

Let
\[
\Omega_\l:=\{x\in M:\,V(x)< \l\},
\qquad
\sigma(\l):=|\Omega_\l|,
\]
where $|\Omega_\l|$ denotes the measure of $\Omega_\l$ associated with the Riemannian metric $g$.
Since $\sigma$ is increasing and left continuous on $(0,\infty)$, we may, consider
$
a\in C([3,\infty))
$
by 
\be\label{defn of a}
a(\l):=\sup\{s\in[0,\infty):2\sigma\big(\l-s\big)\geq \sigma\big(\l+s\big)\}.
\ee
For each $\delta\in(0,1)$, define $d_{\delta}(\l)$ by
\[d_{\delta}(\l):=\sup\Big\{s\in(0,\infty):\l^\n\sigma\big(s\big) \leq\delta \int_{M}\big(\l-V(x)\big)_+^\n\,\dvol_M(x)\Big\}.
\]
Then
\be\label{spectral on d delta is small}\l^\n\sigma\big(d_\delta(\l)\big) \leq \delta\int_{M}\big(\l-V(x)\big)_+^\n\,\dvol_M(x).
\ee

\def\osc{{\mathrm{osc}}}
\def\da{{a}}

\begin{prop}\label{d delta is smaller than lambda}
The following holds.
\begin{itemize}

\item We have
\be\label{spectrum between l-a(l) an l+a(l)}
\big(a(\l)\big)^\n
\bigl|
\Omega_{\l+\delta a(\l)} \setminus \Omega_{\l-\delta a(\l)}
\bigr|
\le
2\int_{\Omega_{\l-a(\l)}} (\l - V(x))_+^\n \, \dvol_M(x).
\ee
\item
$d_\delta(\l) < \l$.
\end{itemize}
\end{prop}

\begin{proof}

First, we compute
\[\ba
&\quad\big(a(\l)\big)^\n
\big|
\Omega_{\l+\delta a(\l)} \setminus \Omega_{\l-\delta a(\l)}
\big|\leq \big(a(\l)\big)^\n
\big|\Omega_{\l+ a(\l)}\big|\leq 2\big(a(\l)\big)^\n\big|\Omega_{\l- a(\l)}\big|\\
&\leq 2\int_{\Omega_{\l-a(\l)}} (\l - V(x))_+^\n \, \dvol_M(x).
\ea
\]
If $d_\delta(\l) \geq \l$, 
\[
 \l^{\n} \sigma\!\big(d_\delta(\l)\big) \geq  \l^{\n} \sigma\!\big(\l\big)\geq \int_{M}\big(\l-V(x)\big)_+^\n\,\dvol_M(x)>\delta\int_{M}\big(\l-V(x)\big)_+^\n\,\dvol_M(x),
\]
which contradicts \eqref{spectral on d delta is small}.
\end{proof}

In this paper, for two positive functions $f,g$ on $\R_+$, 
the notation $f \lesssim g$ means that $\limsup_{\l\to\infty}\frac{f(\l)}{g(\l)}<\infty.$
We write $f \approx g$ if both $f \lesssim g$ and $g \lesssim f$ hold.
\begin{rem}[Motivation for introducing $a(\l)$ and $d_\delta(\l)$]
If  $a(\l) \approx \l$, then $V$ satisfies the doubling condition
\eqref{doubling condition0}.
In contrast, for $V(x) = \ln(|x|)$ with $|x| \ge 1$ on $\R$, which grows
slowly, one has $a(\l) \approx 1$.
Thus, $a(\l)$ measures how the growth type of $V$ deviates from the
doubling condition \eqref{doubling condition0}.

While $d_\delta(\l)$ is motivated by the
Dirichlet-Neumann (DN) bracketing method(see \cref{DN bracketing} for an explanation of this method). 
In \Cref{proof of thm 13}, apply the DN bracketing method,
we show that,
up to an error $\approx \delta \Phi(\l)$,
the contribution from  $\Omega_{d_\delta(\l)}$ can be ignored when estimate $\N(\l)$.
\end{rem}
\paragraph{Oscillation near
$\Omega_{\l+\delta a(\l)} \setminus \Omega_{d_\delta(\l)}$.} For any continuous function $f$ on $M$ and any closed set $U\subset M$, we define the \emph{oscillation} of $f$ on $U$ by
\[
\osc_U(f)\;:=\;\sup\big\{|f(x)-f(y)|:\; x,y\in U\big\}.
\]
By \Cref{d delta is smaller than lambda}, $\Omega_{\lambda+\delta a(\lambda)}\setminus\Omega_{d_\delta(\lambda)}\neq \emptyset,$ we then can set
\be\label{osc radius}
b_\delta(\lambda)
:=\sup\bigl\{\, r\in[0,\mathrm{inj}_x):\ 
\osc_{B_r(x)}(V)\leq\delta^2 a(\lambda),\ 
x\in \Omega_{\lambda+\delta a(\lambda)}\setminus\Omega_{d_\delta(\lambda)}
\bigr\}.
\ee
Here $\mathrm{inj}_x$ denotes the injectivity radius at $x$, and for any $r>0$,
\[
B_r(x):=\{\,y\in M:\ d(y,x)\leq r\,\}.
\]
The first example in \Cref{nontrivial} highlights the natural appearance of the 
injectivity radius in the definition of $b_\delta(\lambda)$ (see also \Cref{rmk main result} and \Cref{compare R and R times S1}).


Let $R_\delta(\lambda)$, $S_\delta(\lambda)$, and $T_\delta(\lambda)$ be
nonnegative quantities defined as the suprema of the norms of the curvature
operator, as well as its first and second covariant derivative, 
respectively, over the $b_\delta(\lambda)$-neighborhood of
\[
\Omega_{\lambda+\delta a(\lambda)} \setminus \Omega_{d_\delta(\lambda)}.
\]
Here, for $r>0$, by the \emph{$r$-neighborhood} of a set $A\subset M$ we mean the open set
$
\bigcup_{x\in A} \mathring{B}_r(x).
$

Let
\[
K_\delta(\l)
:=
R_\delta(\l)
+
S_\delta(\l)^{\frac{2}{3}}
+
T_\delta(\l)^{\frac{1}{2}}.
\]
The exponents are chosen so that $K_\delta(\l)$ is homogeneous under
 metric rescaling.

Let
\be\label{c delta}
c_\delta(\lambda)
:=
\frac{
1
+ K_\delta(\lambda)\, b_\delta(\lambda)^2
}{
a(\lambda)\, b_\delta(\lambda)^2
}.
\ee
We assume for all sufficiently small $\delta>0$,
\begin{equation}\label{assumption}
\lim_{\lambda\to\infty}c_\delta(\l)=0.
\end{equation}

\begin{rem}
   Here $a(\l)$ describes the growth rate of $V$, while $b_\delta^{-1}(\l)$ measures the oscillation scale of $V$. The condition $\lim_{\lambda\to\infty} c_\delta(\lambda)=0$ describes a precise balance between the geometry of the manifold, the growth rate of the potential, and its oscillatory behavior.

     
   
\end{rem}

Our main result is:

\begin{thm}\label{main}
    If \eqref{assumption} holds for all sufficiently small $\delta > 0$, then the classical Weyl law \eqref{classical Weyl asymptotic} holds for the Schr\"odinger operator $\Delta + V$.
\end{thm}

\begin{rem}\label{rmk main result}
It is straightforward to verify that potentials of the form
\[\ba
V(x)&=\ln\!\cdots\!\ln(|x|)\;,|x|\gg1\ \text{on }\mathbb{R}^{n},\qquad\\
V(x,\theta)&=c\cdot(\ln|x|)^a,\quad |x|\gg1,\ a>1,\ (x,\theta)\in \mathbb{R}\times S^1,\\
V(x)&=r(x)^{a}\;,a>1\ \text{on hyperbolic space},
\ea\]
where $r(x)$ denotes the hyperbolic distance from $x$ to $0$ and $c>0$ be  a sufficiently small constant, satisfy \eqref{assumption} (see \cref{examples satisfy assumptions}). None of these potentials satisfy the doubling condition \eqref{doubling condition0}, yet by our theorem, the Weyl law \eqref{classical Weyl asymptotic} still holds for the associated Schrödinger operators.

Conversely, consider the situation in which \eqref{assumption} fails. 
If $\lim_{\lambda\to\infty} a(\lambda)=\infty$ and $(M,g)$ has bounded geometry in the sense of \Cref{defn bounded geometry}, then the failure of \eqref{assumption} only requires $b_\delta(\lambda)$ to be very small, which corresponds to strong oscillation of $V$; see  \cite[\S 6]{rozenbljum1974asymptotics} for examples in which both \eqref{assumption} and the Weyl law \eqref{classical Weyl asymptotic} fail. 
If $\lim_{\lambda\to\infty} a(\lambda)\neq\infty$, then for \eqref{assumption} to hold, strong geometric restrictions must be imposed on $(M,g)$: the curvature must be small, and either the injectivity radius must be large or the oscillation near infinity not too severe.
In the first example of \cref{nontrivial}, $a(\lambda)\to 0$, but the injectivity radius is bounded above due to the $S^1$-factor, so $b_\delta(\lambda)$ cannot be large and \eqref{assumption} fails. In fact, for $V(x,\theta)=c\cdot(\ln|x|)^a$, $(x,\theta)\in\mathbb{R}\times S^1$, and $|x|\gg 1$, we show in \cref{hyperbolic example} that when $a\in(0,1]$, \eqref{assumption} fails and the Weyl law fails.
 When the injectivity radius is infinite, only  the oscillation scale plays a role; in \Cref{hyperbolic example} we show that the potentials
\[
V(x)=r(x)^{a},\qquad 0<a< 1,
\]
on hyperbolic space do not satisfy \eqref{assumption}, and  the classical Weyl law fails. 

Finally, in \cref{example in DY}, we revisit the results of \cite{WeylDY} within the present framework.\end{rem}
\subsection{Organization}
In \Cref{karamata}, we introduce the main ideas and ingredients in the proof of
our main results.
In \Cref{local bounded}, we introduce the notion of local bounded geometry
and develop the associated analytic tools, including heat kernel estimates,
heat kernel expansions, and a quantitative Weyl law, etc.
In \Cref{proof main}, we prove the main result by a rescaling argument,
reducing the problem to the local bounded geometry setting studied in
\Cref{local bounded}.
In \Cref{hyperbolic example}, we present several examples illustrating the
sharpness of our results.
Finally, in \Cref{examples satisfy assumptions}, we discuss examples to which
our main theorem applies, in particular cases where the doubling condition
fails dramatically; we also revisit \cite{WeylDY}, where a doubling
assumption is imposed.

\section{Main ideas and ingredients of the proof}
\label{karamata}
The main theorem proved in this paper relies on very weak assumptions.
As a consequence, the proof is necessarily technical and consists of several
new ideas and key ingredients.
In this section, we briefly list the main ones. Moreover, we believe that these ingredients are of independent interest and
worthy of further study, not only in the context of Weyl's law.

These include pointwise eigenvalue counting functions
(\cref{pointwise eigenvalue}),
Dirichlet-Neumann (DN) bracketing methods
(\cref{DN bracketing}),
local bounded geometry analysis
(\cref{idea of local bounded geometry}),
and a quantitative version of Karamata-Hardy-Littlewood (KHL) Tauberian theorem
(\cref{quantatitative KHL}). In a nutshell, we apply our quantitative KHL Tauberian theorem
to the pointwise eigenvalue counting function together with an explicit model function. The DN bracketing method allows us to localize the problem and to exploit the
oscillation of the potential function. The local bounded geometry analysis,
together with an appropriate rescaling, guarantees that the assumptions of the
quantitative KHL Tauberian theorem are satisfied if $c_\delta(\l)\to 0$.

\subsection{Pointwise eigenvalue counting function}\label{pointwise eigenvalue}

Without the doubling condition, $e^{-tH}$ may not be trace class, and so the heat trace may not exist. However, heat kernel always exists pointwisely and we can study the so-called pointwise eigenvalue counting function as follows.

Let $L$ be a self-adjoint (not necessarily positive) elliptic operator with
discrete spectrum $\{\lambda_j\}_{j=1}^\infty$ on a domain $\Omega$, equipped with either Neumann or Dirichlet boundary conditions, and such that
$\lambda_j \to +\infty$ as $j\to\infty$. Let $\{\phi_j\}_{j=1}^\infty$ be an orthonormal basis of $L^2(\Omega)$
    consisting of eigenfunctions of $L$, $L\phi_j=\l_j\phi_j$.
    The pointwise eigenvalue counting function of $L$ is
    \be\label{pointwise counting for L}
        e_L(\l,x,y)
        := \sum_{\lambda_j< \lambda} \phi_j(x)\phi_j(y),
    \ee
Let $K_L(t,x,y)$ denotes the heat kernel of $L,$ then 
\[K_{L}(t,x,x)=\int_{-\infty}^\infty e^{-tr}d^r e_L(r,x,x),\]
where $d^r$ means differential is take w.r.t. $r$-component.

The following simple statement will play an essential role in the proof.

\begin{lem}\label{lem:spectral-bound-from-heat-kernel}

    For every $x\in \Omega$ and every $\lambda>0$ we have
\begin{equation}\label{eq:pointwise-counting-vs-heat}
        e_L(\l,x,x)\;\leq\; e\, K_L(\lambda^{-1},x,x).
    \end{equation}
\end{lem}
\begin{proof}
   This is because
   \[\ba
e_L(\l,x,x)&=\int_{-\infty}^\l d^r e_L(r,x,x)\leq e\int_{-\infty}^\l e^{-\l^{-1}r}d^r e_L(r,x,x)\\
&\leq e\int_{-\infty}^\infty e^{-\l^{-1}r}d^r e_L(r,x,x)=eK_L(\l^{-1},x,x).\ea
   \]
\end{proof}
\subsection{Dirichlet-Neumann bracketing type argument}\label{DN bracketing}
Assume that $M$ is partitioned into finitely or countably many domains
\be\label{partition}
M=\cup_{j} Q_j,
\qquad
\mathring Q_j\cap \mathring Q_k=\emptyset \quad (j\neq k).
\ee
Let $H_{Q_j,D}$ and $H_{Q_j,N}$ denote the restrictions of $H$ to $Q_j$
with Dirichlet and Neumann boundary conditions, respectively, and let
$\N(\lambda;H_{Q_j,D})$ and $\N(\lambda;H_{Q_j,N})$ be the corresponding eigenvalue
counting functions.
Dirichlet-Neumann (DN) bracketing asserts that
\[
\sum_j \N(\lambda;H_{Q_j,D})
\ \le\
\N(\lambda)
\ \le\
\sum_j \N(\lambda;H_{Q_j,N}).
\]
\paragraph{A key step.}
Using the DN bracketing argument, we find (see
\Cref{reduction to near V=lambda} for details; the argument is elementary in nature, but would be too technical to present at
this point) that, in estimating $\N(\l)$, the contributions of
$\Omega_{d_\delta(\l)}$ and $M \setminus \Omega_{\l-\delta a(\l)}$ to both sides of
the Weyl asymptotic \eqref{classical Weyl asymptotic} can be ignored
(up to an error of order $\approx \delta \Phi(\l)$).
It therefore suffices to focus our analysis on
\[
\Omega_{\l-\delta a(\l)} \setminus \Omega_{d_\delta(\l)}.
\]
Lastly, we emphasize that the local bounded geometry analysis and the
quantitative KHL Tauberian theorem introduced below resolve the essential
difficulties discussed in~\cite[$\S$1.5]{WeylDY} concerning the extension of
the DN bracketing method to general manifolds.

\subsection{Local bounded geometry analysis and the role of $c_\delta(\l)\to 0$}\label{idea of local bounded geometry}

As discussed in \cref{DN bracketing}, in estimating $\N(\l)$, it suffices to focus our analysis on
\[
\Omega_{\l-\delta a(\l)} \setminus \Omega_{d_\delta(\l)}.
\]
\textbf{At this point, the assumption $c_\delta(\l)\to 0$ (i.e. \eqref{assumption}) gets involved.}
This limit allows us, after a suitable rescaling of the metric
(see \cref{rescaling}), to ensure that the geometric data on the region
$\Omega_{\l-\delta a(\l)} \setminus \Omega_{d_\delta(\l)}$ are uniformly bounded.

We introduce in \Cref{local bounded} a notion of local bounded
geometry, and construct a partition of manifolds into domains with good
geometric and analytic properties, summarized in \Cref{nice properties}, which
is suitable for applying the DN bracketing method.

On each such domain, we derive Gaussian-type heat kernel estimates, remainder
estimates for the heat kernel expansion, and a quantitative version of the Weyl
law with uniform control of the constants involved.

\subsection{Quantitative KHL Tauberian theorem}\label{quantatitative KHL}
  In \cite{WeylDY}, a heat kernel approach to Weyl’s law for Schr\"odinger operators
on noncompact manifolds is developed via an extended
Karamata-Hardy-Littlewood (KHL) Tauberian theorem.
In a similar spirit, to derive \eqref{classical Weyl asymptotic} from heat kernel asymptotics without the doubling condition, we need the following  quantitative KHL theorem, whose proof is given in \Cref{Proof of KHL}.

\begin{thm}\label{KHL2}
	Let $\mu$, $\nu$ be positive increasing functions on $\R^+$ and let $ \l_1 > 0$. Suppose there exists a constant $L > 0$ such that for all $t \in(0,\l_1^{-1}]$,
\be\label{exponential bound}
\int e^{-t r} \, d\nu(r) \leq L \, \nu(t^{-1}),
\ee
and the following limit holds uniformly in $s \geq \l_1$:
\be\label{uniform limit of quotient}
\lim_{\tau \to 1} \frac{\nu(\tau s)}{\nu(s)} = 1.
\ee
Let $\b : [0, \infty) \to [0, \infty)$ be a non-decreasing function.
Then the estimate
\be\label{error estimate for laplace transform}
\Big| \int e^{-t r} \, d\mu(r) - \int e^{-t r} \, d\nu(r) \Big|
\leq \b(t) \,  \int e^{-t r} \, d\nu(r), \quad t \in (0, \l_1^{-1}],
\ee
implies that for any $\ep \in (0,1)$, there exist constants $ C_1 = C_1( \ep)>1$ such that 
\be\label{error for measure}
\left|\mu(s) - \nu(s)
\right|
\leq  \big(\ep+\ep L+C_1L\b(C_1 s^{-1})\big) \nu(s),\quad \forall s\geq C_1\l_1.
\ee

\end{thm}
\def\doub{1}
\if\doub0
First, we explain the close relationship between
\eqref{exponential bound} and \eqref{uniform limit of quotient}
and the doubling condition.

    In fact, the doubling condition plays two distinct roles in \cite{WeylDY}.
First, it guarantees that
\[
\int_M e^{-tV(x)}\,\dvol_M(x)<\infty,\qquad \forall\, t>0.
\]
Second, and more importantly, it ensures the validity of
\eqref{exponential bound} and \eqref{uniform limit of quotient},
both of which are essential for deriving the extended KHL Tauberian theorem
in \cite{WeylDY} as well as quantitative KHL here.

More precisely, by imposing the doubling condition on $\nu$, namely that there
exists $\l_0$ such that for all $\l\ge \l_0$,
\[
\nu(2\l)\le C\,\nu(\l)
\]
for some constant $C>0$, it is shown in \cite{WeylDY} that
\eqref{exponential bound} holds whenever $\l\ge (1+\varepsilon_0)\l_0$.
More precisely, there exists $L=L(\varepsilon_0)$ such that
\eqref{exponential bound} is valid.
Moreover, the convolution of $\nu$ with $f_+^\a$ for $\a>0$ satisfies
\eqref{uniform limit of quotient} for any $r\ge (1+\varepsilon_0)\l_0$,
where $f_+^\a(r):=\max\{r^\a,0\}$.
Here the limit in \eqref{uniform limit of quotient} is uniform only in $r$,
and not uniform with respect to $\varepsilon_0$ or $\a$.
\fi

\def\toapply{1}
\if\toapply1
We will apply \Cref{KHL2} to the case where $\mu$ is the pointwise eigenvalue counting function and $\nu$ an explicit model function, see below. It is essential to verify assumptions 
\eqref{exponential bound} and \eqref{uniform limit of quotient}, which will be explained below.
The verification of \eqref{error estimate for laplace transform} follows from
the heat kernel expansion remainder estimate, which is more technical and will
be discussed in \cref{sec of heat expansion remainder}.
\paragraph{Verifying \eqref{exponential bound} and \eqref{uniform limit of quotient}.}This part is technical, but it helps the reader understand the constructions in
\Cref{sec of quantitative weyl law}.
The reader may skip this part on a first reading and return to it when reaching
\Cref{sec of quantitative weyl law}. Another reason for including this part here is that we have just stated
\Cref{KHL2}, and we take this opportunity to explain how its technical
assumptions \eqref{exponential bound} and \eqref{uniform limit of quotient} can be verified.

Unlike \cite{WeylDY}, where the KHL-type theorem is applied globally, here we
apply \Cref{KHL2} in a pointwise manner, in the following sense.

Let $K_H(t,x,y)$ denote the heat kernel of $H$, and let $e(\l,x,y)$ be the
pointwise eigenvalue counting function of $H$.
Formally, one expects the following pointwise asymptotic expansion:
\[
K_H(t,x,x)
=
\int_{\R_+} e^{-tr}\, d^r e(r,x,x)
\sim
(4\pi t)^{-\n} e^{-tV(x)}
=
(2\pi)^{-n}\omega_n
\int_{\R_+} e^{-tr}\, d^r\big(r-V(x)\big)_+^\n .
\]
Here $d^r$ denotes differentiation in the $r$-variable.

\begin{rem}\label{rem laplace transformation of local integral density}
For the last equality above, we have used the fact that for any constant $C>0$,
\be\label{laplace transformation of local integral density}
\int_{\R^+} e^{-tr} (r - C)_+^{\n-1} \, dr
=
\Gamma\Big(\n\Big)\, t^{-\n} e^{-tC},
\ee
and that $\omega_n=\frac{\pi^{n/2}}{\Gamma(\frac{n}{2}+1)}$.
\end{rem}

Set
\[
\mu_x(r) := e(r,x,x),
\qquad
\nu_x(r) := (2\pi)^{-n}\omega_n\big(r-V(x)\big)_+^\n .
\]
As explained in \Cref{DN bracketing}, it suffices to apply \Cref{KHL2} to
$\mu_x$ and $\nu_x$ for
\[
x \in \Omega_{\l-\delta a(\l)} \setminus \Omega_{d_\delta(\l)}.
\]
For this purpose, we need the following simple, though somewhat
technical, statement.

\begin{prop}\label{ob C}
Let $C>0$ and $\nu(r) := (r - C)_+^\a$ with $\a > 0$.
Then for any $\varepsilon_0 > 0$, if $\l_2 \ge (1+\varepsilon_0)C$, we have, for any
$t \in (0,\l_2^{-1}]$,
\be\label{obCeq1}
\int e^{-tr} \, d\nu(r) \le L \nu(t^{-1}),
\ee
where $L$ depends only on $\varepsilon_0$ and $\a$.
Moreover, for $r \ge \l_2$,
\be\label{obCeq2}
\lim_{\tau \to 1} \frac{\nu(\tau r)}{\nu(r)} = 1
\ee
holds uniformly.
\end{prop}

\begin{proof}
There exists a constant $c$, depending only on $\a$
(see also \eqref{laplace transformation of local integral density}), such that
\begin{equation}\label{laplace transformation of local integral density 2}
\int_{\R_+} e^{-tr} \, d\nu(r)
=
c t^{-\a} e^{-tC}
\le
c t^{-\a}.
\end{equation}
For $t \le \l_2^{-1}$, we have
\be\label{verify exponential bound}
\nu(t^{-1})
=
\bigl(t^{-1} - C\bigr)_+^\a
\ge
\Bigl(t^{-1} - \frac{t^{-1}}{1+\varepsilon_0}\Bigr)^\a
=
\Big(\frac{\varepsilon_0}{1+\varepsilon_0}\Big)^\a t^{-\a}.
\ee
Combining \eqref{laplace transformation of local integral density 2}
and \eqref{verify exponential bound}, we obtain \eqref{obCeq1}.
The verification of \eqref{obCeq2} is similar.
\end{proof}

We now apply \Cref{KHL2} when $s=\l$  to estimate $\mu_x(\l)-\nu_x(\l)$.
This requires verifying \eqref{exponential bound} and
\eqref{uniform limit of quotient} for some $\l_1\leq C_1(L,\ep)^{-1}\l \ll \l$.
In view of \Cref{ob C}, this forces $V(x) < (1+\varepsilon_0)^{-1}\l_1 \ll \l$.
Consequently, \eqref{exponential bound} and \eqref{uniform limit of quotient}
can only be verified for $\nu_x$ when $V(x)\ll \l$, which does not cover the
region $\Omega_{\l-\delta a(\l)} \setminus \Omega_{d_\delta(\l)}$ we focus on.

To address this issue, we shift the operator by a suitable constant.
Specifically, we partition the $b_\delta(\l)$-neighborhood of
$\Omega_{\l-\delta a(\l)} \setminus \Omega_{d_\delta(\l)}$ into domains $Q_j$
with $\mathrm{diam}(Q_j)\le b_\delta(\l)$.
On each $Q_j$, set
\[
C_j := \inf_{x\in Q_j} V(x),
\qquad
V_j(x) := V(x) - C_j.
\]
Let $H_{Q_j}$ denote the restriction of $H$ to $Q_j$ with either Dirichlet or
Neumann boundary conditions, and consider
$
\tilde{H}_{Q_j} := H_{Q_j} - C_j,
$
that is, $\tilde{H}_{Q_j}$ is obtained from $H_{Q_j}$ by shifting it by a constant.
Instead of $\nu_x$ and $\mu_x$, we consider the shifted measures
\[
\nu_{j,x}(r) := (2\pi)^{-n}\omega_n\bigl(r - V_j(x)\bigr)_+^{\n},
\qquad
\mu_{j,x}(r) := e_{\tilde{H}_{Q_j}}(r,x,x).
\]
Observe that:
\begin{prop}\label{estimate on Qj}
We have:
\begin{itemize}
\item $0 \le V_j(x) \le \delta^2 a(\l)$ on $Q_j$.
\item $\l - C_j \ge \l - (\l - \delta a(\l)) = \delta a(\l)$.
\end{itemize}
\end{prop}

\begin{proof}
Since $\mathrm{diam}(Q_j) \leq b_\delta(\l)$ and by the definition of $b_\delta(\l)$,
we have $\osc_{Q_j} V \le \delta^2 a(\l)$, which yields the first estimate.
The second estimate easily follows from
$Q_j \cap \Omega_{\l-\delta a(\l)} \neq \emptyset$.
\end{proof}

Note $\nu_{j,x}(\l - C_j) = \nu_x(\l)$, we should apply \Cref{KHL2} to estimate
$\mu_{j,x}(\l - C_j) - \nu_{j,x}(\l - C_j)$. In view of \Cref{ob C} and \Cref{estimate on Qj},
the assumptions \eqref{exponential bound} and
\eqref{uniform limit of quotient} are verified for $\nu_{j,x}$ with $x \in Q_j$,
with
$
\l_1 := 2\delta^2 a(\l),
$
and the constant $L$ depending only on $n$.
Fixing $\ep>0$ and taking $\delta$ sufficiently small, we have
$
\l - C_j \ge 2 C_1(L,\ep)\delta^2 a(\l) = C_1(L,\ep)\l_1.
$
Therefore, we may apply \Cref{KHL2} to estimate
$\mu_{j,x}(\l - C_j) - \nu_{j,x}(\l - C_j)$,
provided that an estimate of the form
\eqref{error estimate for laplace transform} holds.

\subsection{Proof of \Cref{KHL2}}\label{Proof of KHL}

\def\letnu{1}
\if\letnu1

    In the proof, we identify the increasing functions $\mu,\nu$ with the
Stieltjes measures defined by them. The proof is a modification of that of the classical Karamata Tauberian
theorem( c.f. \cite[Theorem~71]{canzani2013analysis}).
 Consider the scaled measures on $\mathbb{R}^+$ defined by, for any set $A\subset\R^+$,
\[
\mu_t(A):=\mu\big(t^{-1}A\big),\qquad
\nu_t(A):=\nu\big(t^{-1}A\big).
\]
 Then for any $l\ge1$ and $\omega=\nu$ or $\mu$,
\begin{equation}\label{intemu}
\int e^{-lr}\,d\omega_t(r)=\int e^{-tlr}\,d\omega(r).
\end{equation}
Hence, by \eqref{intemu} and \eqref{error estimate for laplace transform},
we have for $lt\leq\l_1^{-1}$,
\begin{equation}\label{aes}
\Big|\int e^{-lr}\,d\mu_t(r)-\int e^{-lr}\,d\nu_t(r)\Big|
\le\beta(lt)\int e^{-lr}\,d\nu_t(r)
\le\beta(lt)\int e^{-r}\,d\nu_t(r).
\end{equation}
Consider the space
\[
\mathcal B:=\operatorname{span}
\big\{g_l:\R^+\to\R^+\mid g_l(r)=e^{-lr},\ l\in\{1,2,\dots\}\big\}.
\]
Before proceeding, we briefly explain the idea. Let $t=s^{-1}$. We would like to compare
$
\mu(t^{-1})=\mu_t(1) \text{ and } \nu(t^{-1})=\nu_t(1).
$
Thus the problem reduces to comparing the measures $\mu_t([0,1])$ and $\nu_t([0,1])$. Since the indicator function of $[0,1]$ is not
smooth, we approximate it by smooth bump functions (see $\eta_1$
and $\eta_2$ below). We then use the Stone--Weierstrass theorem to
approximate these bump functions by functions in the class
$\mathcal B$. The integrals of functions in $\mathcal B$ can be estimated
using \eqref{aes}. The conditions \eqref{uniform limit of quotient}
and \eqref{exponential bound} guarantee that the approximation procedure
is uniformly controlled.

Fix $\epsilon\in(0,1)$. By \eqref{uniform limit of quotient}, there exists
$\sigma\in(0,1)$ such that for any $\tau\in(1-\sigma,1+\sigma)$ and
$r\ge\l_1$,
\be\label{nu epsilon tau}
|\nu(\tau r)-\nu(r)|\le\ep\,\nu(r).
\ee
Define smooth non-increasing bump functions
$\eta_1,\eta_2:[0,\infty)\to[0,1]$ such that
\[
\eta_1(r)=
\begin{cases}
1,& r\le1,\\
0,& r\ge1+\sigma/2,
\end{cases}
\qquad
\eta_2(r)=
\begin{cases}
1,& r\le1-\sigma/2,\\
0,& r\ge1.
\end{cases}
\]
By the Stone-Weierstrass theorem, there exist functions
$h_j(\l)=\sum_{l=1}^{k_j}a_{lj}e^{-l\l}\in\mathcal B,j=1,2$
for some $k_j=k_j(\epsilon)$ and coefficients
$a_{lj}=a_{lj}(\epsilon)$ such that for $\l\in[0,\infty)$,
\be\label{Weistrass}\ba
e^{\l}\eta_1(\l)&\le h_1(\l)\le e^{\l}\eta_1(\l)+\epsilon,\\
e^{\l}\eta_2(\l)&\ge h_2(\l)\ge e^{\l}\eta_2(\l)-\epsilon.
\ea\ee
Let
\[
C_1'=\max\{k_1,k_2\}\max\{|a_{lj}|:l=1,2,\dots;j=1,2\},
\qquad
C_2=\max\{k_1,k_2\}+1.
\]
Then it follows from \eqref{intemu}, \eqref{aes} and \eqref{exponential bound} that if $C_2t\leq\l_1^{-1}$, \be\label{eq210} \Big|\int h_j(r)e^{-r}d\mu_t(r)-\int h_j(r)e^{-r}d\nu_t(r)\Big|\leq C_1'\b(C_2t)\int e^{-r}d\nu_t(r)\leq C_1'L\b(C_2t)\nu(t^{-1}). \ee
Then for any $t\le C_2^{-1}\l_1^{-1}$,
\be\ba\label{eq220}
\mu(t^{-1})
&=\mu_t(1)\le \int \eta_1\, d\mu_t
\le \int e^{-r} h_1\, d\mu_t
\quad \text{(by \eqref{Weistrass})}\\
&\le \int e^{-r} h_1\, d\nu_t
   + C_1' L\b(C_2 t)\nu(t^{-1})
\quad \text{(by \eqref{eq210})}.
\ea\ee
Now it suffices to bound $\int e^{-r} h_j\, d\nu_t$ by
$\nu(t^{-1})$, which is easy from the construction of $h_j$:

\begin{lem}\label{L:hj><}
We have, for $j=1,2$, $t<\l_1^{-1}$
\begin{equation}\label{hj<>}
  \bigg|\int e^{-r} h_j\, d\nu_t-\nu(t^{-1}) \bigg|\leq (\ep+\ep L)\nu(t^{-1}).
\end{equation}
\end{lem}
\begin{proof}

Indeed, by \eqref{intemu}, \eqref{Weistrass} and \eqref{exponential bound}, for $j=1,2$
\be\label{eq211}
\Big|\int e^{-r}h_j(r)\,d\nu_t(r)-\int \eta_j(r)\,d\nu_t(r)\Big|
\le \ep\int e^{-r}\,d\nu_t(r)=\ep\int e^{-tr}\,d\nu(r)
\le \ep L\,\nu(t^{-1}).
\ee
Note that
$
\nu\big((1-\sigma/2)t^{-1}\big)
=\nu_t(1-\sigma/2)
\le\int\eta_j(r)d\nu_t(r)
\le\nu_t(1+\sigma/2)
=\nu\big((1+\sigma/2)t^{-1}\big).
$
Hence, by \eqref{nu epsilon tau}, for any $t\le\l_1^{-1}$,
\be\label{eq212}
\Big|\int\eta_j(r)d\nu_t(r)-\nu(t^{-1})\Big|
\le \ep\nu(t^{-1}),\qquad j=1,2.
\ee
\end{proof}

Thus by \eqref{eq220} and \eqref{hj<>}, for any $t\le C_2^{-1}\l_1^{-1}$,\be\label{eq2200}
\mu(t^{-1})\le \big(1+\ep+L\ep+C_1'L\b(C_2t)\big)\nu(t^{-1}).
\ee
Similarly, for any $t\le C_2^{-1}\l_1^{-1}$,
\be\label{eq221}
\mu(t^{-1})
\ge\big(1-\ep-L\ep-C_1'L\b(C_2t)\big)\nu(t^{-1}).
\ee
By \eqref{eq2200} and \eqref{eq221}, setting $s=t^{-1}$ and $C_1=\max\{C_1',C_2\}$ finishes the proof.

		\fi
\section{Local bounded geometry analysis}\label{local bounded}
In this section, we assume that $(X,g)$ is a complete Riemannian
manifold with $\dim(X)=n$.
Rather than imposing bounded geometry assumptions on the entire space,
we impose them only on an open subset of $(X,g)$. Our goal is to
construct a nice partition of $X$, analogous to the standard cube
partition of $\R^n$.
In this section, we show that such a partition exists (\Cref{vitali}), and that each piece of the
partition enjoys good geometric and analytic properties, as summarized in
\Cref{nice properties}.
These properties allow us to obtain uniform heat kernel estimates, uniform heat
kernel expansion remainder estimates (\Cref{sec of heat expansion remainder}),
and a uniform quantitative Weyl law (\cref{sec of quantitative weyl law}).
In this way, we are able to extend the DN bracketing method to
 manifolds.
 
\def\stum{1}
\if\stum1
\subsection{Review on nice domains}\label{nice properties}

In this subsection, for the reader’s convenience, we briefly review 
several classes of domains with good geometric and analytic properties, which can
be found in standard textbooks or monographs on geometric analysis
\cite{adams2003sobolev,jost2005riemannian,grigoryan2009heat}, as well as in
Grigor'yan's papers~\cite{grigor1994heat,grigor1997gaussian}.
\def\WBG{{\mathrm{WBG}}}
\begin{defn}[Weak bounded geometry $\WBG(R,\tau)$]
    Let $R,\tau>0$.  We say that a domain $U\subset X$ belongs to $\WBG(R,\tau)$
 if the absolute value of the Ricci curvature is
bounded by $(n-1)R$, and the injectivity radius is
uniformly bounded below by $\tau>0$ at every $x\in U$.
\end{defn}
\begin{defn}[Lipschitz domain $\Lip(L_1,L_2,r_1)$]\label{lip domain}
Let $L_1,L_2>1$ and $r_1>0$. We say that a domain $U\subset X$ belongs to
$\Lip(L_1,L_2,r_1)$ if for every $x\in\partial U$ there exists a coordinate chart
$\varphi:B_{r_1}(x)\to\R^n$ and a Lipschitz function
$\psi:\R^{n-1}\to\R$ such that
\[
\varphi(B_{r_1}(x)\cap U)
=\{(y',y_n)\in\varphi(B_{r_1}(x)):\ y_n>\psi(y')\},
\qquad
\varphi(B_{r_1}(x)\cap\partial U)
=\{y_n=\psi(y')\}.
\]
Moreover, $\varphi$ is bi-Lipschitz and $\psi$ is Lipschitz, with
\begin{equation}\label{Lip varphi0}
L_1^{-1}d(z,w)\le |\varphi(z)-\varphi(w)|\le L_1 d(z,w),
\qquad z,w\in B_{r_1}(x),
\end{equation}
and
\begin{equation}\label{lip psi0}
|\psi(y')-\psi(z')|\le L_2|y'-z'|,
\qquad y',z'\in\R^{n-1}.
\end{equation}
\end{defn}
\def\Cone{{\mathrm{Cone}}}
\begin{defn}[Uniform interior cone $\Cone(L,\rho,\varepsilon_0,\theta)$] \label{defn cone}
Let $L>1$, $\rho,\varepsilon_0>0$, and $\theta\in(0,\pi/2]$. A domain
$U\subset X$ belongs to $\Cone(L,\rho,\varepsilon_0,\theta)$ if for every
$x\in\partial U$ there exists a coordinate chart
$\varphi:B_\rho(x)\to\R^n$ as in \Cref{lip domain} with bi-Lipschitz constant $L$
such that
\[
y+z\in \varphi(B_\rho(x)\cap U)
\]
when $y\in \varphi( B_{\rho/2}(x)\cap \bar{U})$ and $z\in C$, where
\[
C:=\bigl\{(z',z_n):(\cot\theta)|z'|<z_n<\varepsilon_0\bigr\}.
\]
\end{defn}

\def\VD{{\mathrm{VD}}}
It is standard that the following holds.

\begin{prop}\label{lip is cone}
If $U\in \Lip(L_1,L_2,r_1)$, then $U\in \Cone(L,\rho,\varepsilon_0,\theta)$ for some
constants $L,\rho,\varepsilon_0,\theta$ depending only on $L_1,L_2,r_1$ and $n$.
\end{prop}\begin{proof}
    It is easy to check that we can take $L=L_1,\rho=r_1,\varepsilon_0=\frac{r_1}{2L_1\sqrt{n}},\theta=\arctan(L_2^{-1}).$
\end{proof}
\begin{rem}
 In fact, it is well known that the uniform Lipschitz condition is equivalent to
the uniform cone condition, which consists of both uniform interior and uniform
exterior cone conditions.
However, the uniform interior cone condition alone does not guarantee uniform Lipschitz
regularity; see the figure below for a counterexample.
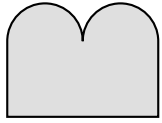
\begin{center}
\begin{tikzpicture}
  \def\r{0.5}          
  \def\fillc{gray!25}  

  \filldraw[fill=\fillc, draw=black, thick]
    (0,0) -- (2,0) -- (2,1)
    arc[start angle=0, end angle=180, radius=\r]  
    arc[start angle=0, end angle=180, radius=\r]  
    -- (0,0) -- cycle;
\end{tikzpicture}
\end{center}
\end{rem}

\begin{defn}[Global Sobolev inequality $S(C_S)$]
Let $C_S>0$. A  domain
$U\subset X$ belongs to $S(C_S)$ if for all $u\in W^{1,2}(U),$
$$
\|u\|_{L^{\frac{2n}{n-2}}}\leq C_S\| u\|_{W^{1,2}}.
$$
\end{defn}

The following proposition is well known. 
\begin{prop}\label{Cone implies SObolev}
     If $U\in \Cone(L,\rho,\varepsilon_0,\theta)\cap\WBG(R,\tau)$,  then $U\in S(C_S)$ for some $C_S$ depending only on $L,\rho,\varepsilon_0,\theta,R,\tau$ and $n$.
\end{prop}
\begin{proof}[Outline of proof]
The proof is essentially the same as in the $\R^n$ case treated in
\cite[Theorem~4.12]{adams2003sobolev}, where uniform control of the
constants is achieved.
The key point is: if $u\in C^\infty(\R^n)$, then for any $x\in\R^n$ and any
cone $C_x$ with vertex at $x$, one has (see \cite[Lemma~4.15]{adams2003sobolev})
\be
|u(x)| \le |u_{C_x}| + C'\int_{C_x} |\nabla u(y)|\,|x-y|^{1-n}\,dy,
\ee
where $u_{C_x}$ denotes the average of $u$ over $C_x$, and the constant $C'$
depends only on $n$ and on the height and opening angle of the cone $C_x$.

An analogous estimate can be obtained on $U$ by reducing the problem to the
Euclidean case: Near the boundary, one uses coordinate charts $\varphi$ as in
\Cref{lip domain}; away from the boundary, one uses the fact that the exponential
map is bi-Lipschitz on balls of radius $r\le r_*$, where the admissible radius
$r_*>0$ and the uniform bi-Lipschitz constants depend only on $(R,\tau)$
(cf.~\cite[Corollary~6.6.1]{jost2005riemannian}). As long as the size of the cone is smaller than a radius
$r_*$, the volume of
the cone, as well as the integral of $d(x,y)^{1-n}$ over a cone or ball of radius
at most $r_*$, is uniformly comparable to the corresponding  quantities in $\R^n$.
With these observations in place, the proof follows by repeating the argument of
\cite[Theorem~4.12]{adams2003sobolev} verbatim.
\end{proof}

\begin{defn}[Uniform Gaussian estimate $G(c_1)$]
    Let $c_1>0$, and let $U\subset X$ be a domain. 
Let $K(t,x,y)$ denote either the Dirichlet or Neumann heat kernel on $U$
(for the Laplace--Beltrami operator associated with $g$). We say  $U\in G(c_1)$,
if for all $t\in(0,1)$,
\begin{equation}\label{heat kernel estimate on Qj0}
0\le K(t,x,y)
\le c_1\, t^{-\n}e^{-\frac{\, d^2(x,y)}{8t}},\quad \text{ a.e. } x,y\in U.
\end{equation}

\end{defn}
\def\stu{1}
\if\stu0
The result below is a special case of~\cite[Theorem~4.1]{sturm1996analysis}.
Sturm’s work in fact yields two-sided Gaussian heat kernel estimates; however, we state only the version we need.

\begin{thm}[Sturm's theory, Theorem 4.1 in \cite{sturm1996analysis}]\label{heat kernle estimate0}
If $U\in\VD(r_0,C_D)\cap \WP(r_0,C_P,E)$, then $U\in G(c_1,c_2,c_3)$ for some constants $c_1,c_2,c_3$ depending only on $r_0,C_D,C_P,n$ and $E$. 
\end{thm}

It follows from \Cref{lip and geometry bounds implies poincare and VD} and \Cref{heat kernle estimate0} that
\begin{cor}\label{cor heat}
    If $U\in \Lip(L_1,L_2,r_1)\cap\WBG(R,\tau)$, then
$U\in G(c_1,c_2,c_3)$ for some constants $c_1,c_2,c_3$ depending only on
$L_1,L_2,r_1,R,\tau$, and $n$.

\end{cor}
\fi

The statement below is well known, and there are several ways to establish it;
see \cite{grigor1994heat,grigor1997gaussian} for references to different
approaches. For the reader’s convenience, we present the approach based on
Varopoulos~\cite{varopoulos1985hardy}, combined with Grigor'yan’s method
\cite{grigor1997gaussian}.
 \begin{prop}\label{Varo}
Let $U\in S(C_S)$ be bounded. Then $U\in G(c_1)$ for some $c_1$ depends only on $C_S$ and $n.$
\end{prop}
\begin{proof}[Outline of proof]
All constants here depend only on $C_S$ and $n$. Moreover, by the maximal principle, we only need to deal with the Neumann heat kernel.
\paragraph{Varopoulos' argument.} First, we carry the Varopoulos' (Sobolev $\Rightarrow$ Nash
$\Rightarrow$ ultracontractivity)  argument \cite{varopoulos1985hardy}.
Recall that the Neumann Laplacian $\Delta$ on $U$ is the self-adjoint operator
associated with the closed quadratic form
\[
\mathcal E(f,f):=\int_{U}|\nabla f|^2\,dx,
\qquad \Dom(\mathcal E)=W^{1,2}(U).
\]
Since $U$ is bounded, we have $1\in \Dom(\mathcal E)$. Thus
$e^{-t\Delta}$ is conservative, which implies
\[
\|e^{-t\Delta} f\|_{L^1(U)}=\|f\|_{L^1(U)},\qquad
0\le f\in L^2(U)\cap L^1(U).
\]
Consequently,
\be\label{eq:mass_conservation_neumann}
\|e^{-t\Delta} f\|_{L^1(U)}\le \|f\|_{L^1(U)},\qquad
f\in L^2(U)\cap L^1(U).
\ee
Moreover, the Sobolev inequality on $U$ implies the Nash inequality
\begin{equation}\label{eq:nash_critical}
\|f\|_{L^2(U)}^{2+4/n}
\le C_N\,
\Big(\|\nabla f\|_{L^2(U)}^2+\|f\|_{L^2(U)}^2\Big)\,
\|f\|_{L^1(U)}^{4/n},
\qquad f\in W^{1,2}(U)\cap L^1(U).
\end{equation}
 Let $u(t):=e^{-t\Delta}f$. Then by \eqref{eq:mass_conservation_neumann} and \eqref{eq:nash_critical} (with $c=C_N^{-1}$),
\[
\frac{d}{dt}\|u(t)\|_{L^2(U)}^2
=-2\|\nabla u(t)\|_{L^2(U)}^2
\le
-c\,\frac{\|u(t)\|_{L^2(U)}^{2+4/n}}{\|f\|_{L^1(U)}^{4/n}}+2\| u(t)\|_{L^2(U)}^2,
\]
which yields
$
\|e^{-t\Delta} f\|_{L^2(U)}
\le C e^{2t}\,t^{-n/4}\|f\|_{L^1(U)},\, t>0.
$
By self-adjointness, \[ \|e^{-t\Delta}\|_{L^1(U)\to L^\infty(U)}\leq\|e^{-t\Delta/2}\|_{L^1(U)\to L^2(U)}\|e^{-t\Delta/2}\|_{L^2(U)\to L^\infty(U)}\le C_0\,t^{-n/2},\, t\in(0,2). \] Hence for the Neumann heat kernel, \be\label{Varopu} 0\leq K(t,x,y)\le C_0\,t^{-n/2},t\in(0,2).\ee
\paragraph{On-diagonal estimates \eqref{Varopu} imply Gaussian estimates.}
There are several ways to establish this implication; here we present the
approach due to Grigor'yan~\cite{grigor1997gaussian}. Define
\[
E(x,t):=\int_U K^2(t,x,y)e^{\frac{d^2(x,y)}{4t}}\,dy .
\]
Proceeding exactly as in the proof of \cite[Theorem~2.1]{grigor1997gaussian}, one
can show that, since
\be\label{upper bound of I}
I(x,t):=\int_U K^2(t,x,y)\,dy
=K(2t,x,x)\le C_0'\,t^{-\n},\qquad t\in(0,2),
\ee
there exists a constant $C_1$ such that
\be\label{upper bound of ED}
E(x,t)\le C_1\,t^{-\n},\qquad t\in(0,1).
\ee
Here the last inequality of \eqref{upper bound of I} is guaranteed by \eqref{Varopu}.
 Although \cite{grigor1997gaussian} assumes that \eqref{upper bound of I} holds for
all $t>0$, a direct inspection of the proof shows that it suffices to assume
\eqref{upper bound of I} for $t\in(0,2)$, which yields the Gaussian estimate for
$t\in(0,t_0)$ with any $t_0<2$.

Using the inequality below, (which follows easily from the Cauchy--Schwarz
inequality together with the estimate
$d^2(x,z)+d^2(y,z)\ge d^2(x,y)/2$, see \cite[Proposition~5.1]{grigor1994heat})
\[
K(t,x,y)\le
\sqrt{E\big(x,t/2\big)\,E\big(y,t/2\big)}
\exp\Big(-\frac{d^2(x,y)}{8t}\Big), t>0
\]
and \eqref{upper bound of ED}, the desired Gaussian estimate follows.

\end{proof}

Lastly, we introduce the concept needed for the study of the heat kernel
expansion.
\begin{defn}[Uniform tube condition $T(c,\epsilon_0)$]\label{tube}
Let $c>0$ and $\epsilon_0>0$.  A bounded domain $U\subset X$
belongs to $T(c,\epsilon_0)$, if
for every $0<\epsilon<\epsilon_0$,
\[
|\{x\in U:\ d(x,\partial U)\le \epsilon\}|\le c\,\epsilon|U|^{\frac{n-1}{n}}.
\]
\end{defn}
    
\fi

\subsection{Voronoi-like tessellation}\label{vitali}
Let $W \subset X$ be a relative compact open subset. In this subsection, we assume that there exist constants $r_0, R$ such that:
\begin{enumerate}[(i)]
    \item\label{local bounded 1} The injectivity radius at every point $x \in W$ is bounded below by $3r_0$.
    \item\label{local bounded 2} We assume that the absolute value of the Ricci curvature on the $3r_0$-neighborhood of $W$
    is bounded above by $(n-1)R$.
\end{enumerate}
\def\tQ{{\tilde{Q}}}

The lemma below is used to construct a good ``tessellation'' of $2r_0$-neighborhood of  $W$.

\begin{lem}[Vitali Covering Lemma]\label{Vitali covering0}
If \eqref{local bounded 1} and \eqref{local bounded 2} hold, then the
$2r_0$-neighborhood of $W$ can be covered by finitely many balls
$\{\mathring{B}_{r_0}(x_j)\}_{j=1}^I$ such that:
\begin{enumerate}[(1)]
    \item
    Any two balls in the collection
    $\{\mathring{B}_{5^{-1} r_0}(x_j)\}_{j=1}^I$ are disjoint.
    \item
    There exists a constant
    \[
    N>0,
    \]
    depending \textbf{only} on $n,r_0,R$, such that each ball
    $\mathring{B}_{r_0}(x_j)$ intersects at most $N$ other balls.
\end{enumerate}
\end{lem}

\def\Vol{\mathrm{Vol}}

\begin{proof}
The existence of a covering satisfying \emph{(1)} follows from the Vitali
covering lemma (cf.\ \cite[\S1.3]{fanghua2003geometric}). Fix such a covering.
If $\mathring{B}_{r_0}(x_l)$ intersects $\mathring{B}_{r_0}(x_j)$, then
$d(x_j,x_l)\le 2r_0$. Hence
$
\sqcup_{\{\,l:\,\mathring{B}_{r_0}(x_l)\cap \mathring{B}_{r_0}(x_j)\neq\emptyset\,\}}
\mathring{B}_{5^{-1}r_0}(x_l)
\subset \mathring{B}_{3r_0}(x_j).
$
By the Bishop--Gromov volume comparison theorem, it follows that one may take
$
N=\frac{\Vol_{-R}(3r_0)}{\Vol_{R}\big(5^{-1}r_0\big)}.
$
Here, for any $k\in\R$ and $r>0$, $\Vol_k(r)$ denotes the volume of the geodesic
ball of radius $r$ in the space form $M^n_k$.
\end{proof}

Fix a covering $\{ \mathring{B}_{ r_0}(x_j) \}_{j=1}^{I}$ of \textbf{$2r_0$-neighborhood} of $W$ satisfying \Cref{Vitali covering0}. Mimicking the construction of Voronoi cells, we now construct a ``tessellation’’  as follows:
\[
\tQ_j := \bigl\{ x \in {B_{r_0}(x_j)} : d(x, x_j) \leq d(x, x_l), \ l \neq j \bigr\}, 1\leq j\leq I.
\]
Here $d$ denotes the distance with respect to the metric $g$. See the figure below for the case $I=2$, where the distance between the two
centers is smaller than $2r_0$.

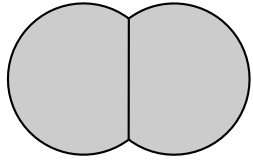
\begin{center}
\begin{tikzpicture}

    \path[name path=circle1] (0,0) circle (1);
    \path[name path=circle2] (1.2,0) circle (1);

    \path [name intersections={of=circle1 and circle2, by={A, B}}];

    \filldraw[fill=gray!40, thick] 
        let \p1 = ($(A)-(0,0)$), \n1 = {atan2(\y1,\x1)},
            \p2 = ($(B)-(0,0)$), \n2 = {atan2(\y2,\x2)},
            \p3 = ($(B)-(1.2,0)$), \n3 = {atan2(\y3,\x3)},
            \p4 = ($(A)-(1.2,0)$), \n4 = {atan2(\y4,\x4)}
        in
        (A) arc (\n1:\n2+360:1) -- 
        (B) arc (\n3:\n4:1) -- cycle;

    \draw[thick] (A) -- (B);
\end{tikzpicture}
\end{center}
The statement below is easy to check.
 A proof is included in $\S$\ref{appendix} for the reader's convenience.
\begin{prop}\label{tesslation}
 If we further assume that the absolute value of the sectional curvature on the $3r_0$-neighborhood of $W$
    is bounded above by $R$, and $4r_0\leq \pi/\sqrt{R}$.
Then the tessellation $\{\tQ_j\}_{j=1}^{J}$ enjoys the following properties:
\begin{enumerate}[(1)]
    \item The interiors of any two distinct $\tQ_i$ and $\tQ_j$ are disjoint.
    \item We have $B_{5^{-1}r_0}(x_j)\subset \tQ_j\subset {B_{r_0}(x_j)}.$  Moreover, there exist constants $c_4,c_5>0$ depending \textbf{ only } on $n$ and $R$ such that
    \[
    c_4r_0^n \;\le\; |\tQ_j| \;\le\; c_5r_0^n.
    \]
    \item $\tQ_j\in\Lip(L_1,L_2,r_1)$ for some $L_1,L_2$ and $r_1$ depending \textbf{only} on $n,R$ and $r_0$. 
\item\label{tess item 4} $\tQ_j\in T(c,\epsilon_0)$ (see \Cref{tube}) for some $c, \epsilon_0$ depending only on $n$, $R$ and $r_0$.
\item $\tQ_j\in G(c_1)$ for some constants $c_1$ depending only on $n$, $R$ and $r_0$.
\end{enumerate}
\end{prop}
We now construct a ``{tessellation}'' of $X$ based on the ``tessellation'' of $2r_0$-neighborhood of $W$ above. A naive approach would be to consider
$$
\tQ_0=\overline{\,X\setminus \cup_{j=1}^I \tQ_j\,}
$$ together with $\tQ_j,1\leq j\leq I.$
However, in  this construction, $\tQ_0$ is not in general a domain with a well-behaved boundary. Indeed, the situation illustrated below may occur: after removing the three gray balls indicated in the figure, the remaining set fails to be a Lipschitz domain and does not satisfy any uniform  interior cone condition.

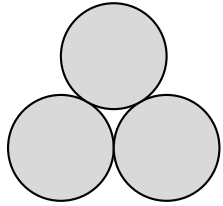
\begin{center}
\begin{tikzpicture}
  \def\r{0.7} 
  \def\gray{gray!30} 

  \coordinate (A) at (0,0);
  \coordinate (B) at (2*\r,0);
  \coordinate (C) at (\r,{sqrt(3)*\r});

  \begin{scope}[shift={(-\r,{-sqrt(3)*\r/3})}]

    \filldraw[fill=\gray, draw=black, thick] (A) circle (\r);
    \filldraw[fill=\gray, draw=black, thick] (B) circle (\r);
    \filldraw[fill=\gray, draw=black, thick] (C) circle (\r);

  \end{scope}

\end{tikzpicture} 
\end{center}

We now carry out the following construction instead.
To simplify the explanation, we may assume that there exists $J$ such that
\[
\tilde Q_j \cap \tilde Q_0 \neq \emptyset
\quad \text{if and only if} \quad j>J.
\]

We then set
\[
D_0 := \tilde Q_0 \cup \big(\cup_{j>J} \tilde Q_j\big),
\qquad
Q_j := \tilde Q_j,\; 1\le j\le J.
\]
We have the following result.
\begin{prop}\label{properties of tessellation0}
Under the same assumptions as in \Cref{tesslation},
\begin{enumerate}[(1)]
    \item
    $X=D_0\cup(\cup_{j=1}^{J} Q_j)$, and their interior are disjoint.
    
    \item
    Each $Q_j$, $1\le j\le J$, satisfies all the properties listed in
    \Cref{tesslation}.
    
    \item
    $D_0\subset X\setminus W $ consists of domains belonging to
    $\Cone(L,\rho,\varepsilon_0,\theta)$ for some constants
    $L,\rho,\varepsilon_0,\theta$ depending only on $n$, $R$, and $r_0$.
\end{enumerate}
\end{prop}

\begin{proof}
Only item~(3) requires justification. Since $\{\tQ_j\}_{j=1}^I$ covers the $2r_0$-neighborhood of $W$, we have
$\mathrm{dist}(\tQ_0,W)>2r_0$. Moreover, by the construction of $D_0$, it is contained
in the $2r_0$-neighborhood of $\tQ_0$. Hence
$
D_0\subset X\setminus W.
$
By construction of $D_0$, each  $x\in \partial D_0$ belongs to some
$\tilde Q_j \subset D_0$ with $j> J$.
Since each $\tilde Q_j$ admits a uniform interior cone by \Cref{lip is cone} and item (3) in \Cref{tesslation}, the same
uniform interior cone condition holds for $D_0$.
\end{proof}

\def\tQ{\tilde{Q}}

\subsection{Heat kernel expansion remainder estimate}\label{sec of heat expansion remainder}

In this subsection, we derive a heat kernel expansion remainder estimate (\Cref{heat remainder estimate}).

\def\BG{{\mathrm{BG}}}
\def\K{{\mathcal{K}}}
\begin{defn}[Bounded geometry $BG(\K,r_0)$]
    Let
$U \subset X$ be a bounded domain. We say $U$ has bounded geometry, write $U\in \BG(\K,r_0)$,
if the norm of the curvature operator, as well as that of its first two
    covariant derivatives, is bounded above on  $U$ by some constant
    $\mathcal{K}>0$.
    Moreover, the injectivity radius at every $x \in U$ is  bounded
    below by a constant $r_0>0$.
\end{defn}
Assume $U\in\BG(\K,r_0)$ for some constant $\K$ and $r_0$.
Let $H := \Delta + V$ for some potential $0\leq V \in C(X)$, and let
$K_{H}$ denote either the Dirichlet or Neumann heat kernel of $H$ on $U$.
Let $k^0_{H}$ be an approximation of the heat kernel $K_{H}$, given by
\[
k^0_{H}(t,x,y)
=
(4\pi t)^{-\n}
e^{-t V(x)}
\exp\!\Big(-\frac{d^2(x,y)}{4t}\Big)
G^{-\frac14}(x,y),
\qquad x,y \in U.
\]
Here $G := \det(g_{ij})$, where $g_{ij}$ is the metric expressed in normal
coordinates centered at $x$  and $d$ denotes the distance induced by $g$.

\begin{prop}\label{Remainder R}
Let
\[
R(t,x,y) := (\partial_t + H) k^0_{H}(t,x,y).
\]
Then near the diagonal,
\[
R(t,x,y)
=
(4\pi t)^{-\n}
e^{-t V(x)}
\exp\!\Big(-\frac{d^2(x,y)}{4t}\Big)
\Big(
\Delta G^{-1/4}
+
\big(V(y)-V(x)\big)
\Big).
\]
Here $H$ and $\Delta$ act on the $y$-variable.
\end{prop}

\begin{proof}
This follows from a straightforward computation, as in
\cite[Proposition~3.4]{WeylDY}.
\end{proof}
\def\K{{\mathcal{K}}}

Using the Jacobi field equation (as well as its differentiated versions),
it is standard to show that there exist constants
\be\label{r one} r_* = c(\K) < r_0\text{ and }C = C(n,\K)\ee such that
\be\label{derivative of G}
|\Delta G|(x,y) \le C
\quad\text{and}\quad
|\nabla G|(x,y) \le C
\ee
whenever $d(x,y) \le r_*$, for all $x \in U$.
Here all derivatives act on the $y$-variable.

Choose $\phi \in C_c^\infty(\R)$ to be a bump function such that
\[
\phi(s)=1 \ \text{for } |s|\le \tfrac12,
\qquad
\phi(s)=0 \ \text{for } |s|\ge \tfrac34.
\]
Let $\epsilon_0 <\min\{\mathrm{diam}(U), r_*,1\}$ be a fixed positive constant.
Assume that $x \in U$ satisfies $d(x,\partial U) \ge \epsilon_0$.
Set
\[
\varphi(x,y) := \phi\Big(\frac{d(x,y)}{\epsilon_0}\Big),
\]
and consider
\[
K^0_{H}(t,x,y)
:=
\varphi(x,y)\, k^0_{H}(t,x,y)
=
(4\pi t)^{-\n}
e^{-t V(x)}
\exp\!\Big(-\frac{d^2(x,y)}{4t}\Big)
G^{-\frac14}(x,y)\, \varphi(x,y).
\]
By a straightforward computation, we obtain the following result.
\def\tR{{\tilde{R}}}
\begin{prop}\label{tilde R}
Let
\[
\tR(t,x,y) := (\partial_t + H) K^0_{H}(t,x,y).
\]
Then
\[
\tR(t,x,y)
=
\varphi(x,y) R(t,x,y)
+
\Delta \varphi(x,y)\, k^0_{H}(t,x,y)
-
2 \big(\nabla \varphi(x,y), \nabla k^0_{H}(t,x,y)\big).
\]
Here all derivatives act on the $y$-variable.
\end{prop}
Note that
\begin{itemize}
\item the support of $\Delta\varphi$ and $|\nabla\varphi|$ is contained in
$\{y:d(x,y)>\ep_0/2\}$;
\item if $d(x,y)\ge\ep_0/2$ and $t\in(0,\ep_0^{3})$, then
$
\ep_0^{-2}t^{-2}e^{-\frac{d^2(x,y)}{8t}}\le C
$
for some universal constant $C>0$. Indeed, for $N>2$, let
$S_N:=\sup_{t>0}t^Ne^{-t}$, then
$
\ep_0^{-2}t^{-2}e^{-\frac{d^2(x,y)}{8t}}
\le \ep_0^{-2}t^{-2}e^{-\frac{\ep_0^2}{32t}}
\le S_N\ep_0^{-2}t^{-2}(32t\ep_0^{-2})^N
\le S_N32^N\ep_0^{3(N-2)-2-2N}.
$
Choosing $N=8$,
we may take $C=32^8S_8$.
\end{itemize}
 It follows immediately from \Cref{tilde R} and \Cref{Remainder R} that for some $C=C(n,\K)>0$,
\be\label{remainder}
|\tR|(t,x,y)
\le
C\, t^{-\n}
e^{-\frac{d^2(x,y)}{8t}}
\big(1+|V(x)-V(y)|\big)\chi(x,y),
\qquad
t \in (0,\epsilon_0^3),
\ee
where $\chi(x,y)=1$ if $d(x,y)<\epsilon_0$ and $\chi(x,y)=0$ otherwise.

We have the following estimate.

\begin{thm}\label{heat remainder estimate}
Assume that $U\in G(c_1)\cap\BG(\K,r_0)$ for some constants $c_1,\K,r_0>0$.
Fix $\epsilon_0<\min\{\mathrm{diam}(U),r_*,1\}$ (see \eqref{r one} for the definition of $r_*$). If
\[
t\in\bigl(0,\min\{1,\epsilon_0^3\}\bigr)
\qquad\text{and}\qquad
d(x,\partial U)\ge \epsilon_0,
\]
then
\be\label{remainder estimate eq1}
\bigl|K_{H}(t,x,x)-(4\pi t)^{-\n}e^{-tV(x)}\bigr|
\le
C\,t^{-\n+1}\max\{1, \sup_U(V)\}.
\ee
where $C$ depends only on $c_1,n$, and $\K$.

\end{thm}

\begin{proof}
Let $K$ denote either the Dirichlet or Neumann heat kernel on $U$, then we have, by Duhamel principle,
\be\ba\label{duhamel}
&\quad
|K_{H}(t,x,x)-K_{H}^0(t,x,x)|
\le
\int_0^t \int_U
|K_{H}(s,x,y)|\,
|\tR(t-s,x,y)|\,
\dvol(y)\, ds \\
&\le
\int_0^t \int_U
|K(s,x,y)|\,
|\tR(t-s,x,y)|\,
\dvol(y)\, ds \\
&\le C\max\{1, \sup_U(V)\}\int_0^t \int_{{\{y : d(x,y) < \epsilon_0\}}}
s^{-\n}(t-s)^{-\n}e^{-\frac{d^2(x,y)}{8s}}e^{-\frac{d^2(x,y)}{8(t-s)}}
\dvol(y)\, ds .
\ea\ee
Here the first inequality follows from the Duhamel's principle, the second from 
$0 \le K_{H} \le K$, and the third  from \eqref{remainder}, 
$U\in G(c_1)$.

Note that for $0 < s < t$,
\[
\frac{d^2(x,y)}{s} + \frac{d^2(x,y)}{t-s}
=
\frac{t\, d^2(x,y)}{s(t-s)}.
\]
Moreover, curvature bounds imply that there exists
$C = C(\K,n)$ such that
\[
\int_{\{y : d(x,y) < \epsilon_0\}}
\exp\!\Big(-\frac{d^2(x,y)}{t}\Big)\,
\dvol(y)
\le C t^\n,\forall t>0.
\]
The proposition then follows from \eqref{duhamel} and the fact that $K_{H}^0(t,x,x)={(4\pi t)^{-\n}}e^{-tV(x)}
$.
\end{proof}

\subsection{Quantitative Weyl's law}\label{sec of quantitative weyl law}
To apply the DN bracketing method, we need a quantitative version
of Weyl's law as follows.
\begin{thm}\label{quantative Weyl law}
   Let $U \in \BG(\K,r_0)\cap G(c_1)\cap T(c,\ep_0)$  with
$\ep_0<\min\{\mathrm{diam}(U),r_0,1\}$. Let $0\le V\in C(\bar U)$ and let
$H:=\Delta+V$ be the Schr\"odinger operator with either Dirichlet or Neumann
boundary condition. Let $S$ be a constant such that 
$
S\geq\max\{\sup_{x\in U} V(x),1\}.
$
Let $\N(s,H)$ denote the eigenvalue counting function of $H$. Then for any
$\ep>0$ small enough, there exist constants
$\cc=\cc\big(\ep,n\big),\cc'=\cc'(\ep,\K,r_0,c_1,n),\cc''=\cc''(c,c_1,n)$ such that for any
$s>\cc(2S+\cc)$, one has
\be\ba
\N(s,H)
&\le
\Big(1+\cc''\ep+\cc'Ss^{-1}\Big)(2\pi)^{-n}\omega_n
\int_U (s-V(x))_+^\n\,\dvol_X(x),
\\
\N(s,H)
&\ge
\Big(1-\cc''\ep-\cc'Ss^{-1}\Big)(2\pi)^{-n}\omega_n
\int_{Q_j} (s-V(x))_+^\n\,\dvol_X(x).
\ea\ee

\end{thm}
\begin{proof}
Set
\[
U^1:=\{x\in U:d(x,\partial U)\geq \ep\},
\qquad
U^2:=U\setminus U^1.
\]
Let $K_{H}$ (resp. $K$) be the either Dirichlet or Neumann heat kernel for $H$ (resp. $\Delta$).

Using \Cref{lem:spectral-bound-from-heat-kernel}, the fact that $U\in G(c_1)\cap T(c,\ep_0)$, we obtain if $s>2S$, $\ep<\ep_0,$
\begin{equation}\label{Weyl law U2}
\begin{aligned}
&\quad\int_{U^2} e_{H}(s,x,x)\,\dvol_X(x)
\le e \int_{U^2} K_{H}\!\left(s^{-1},x,x\right)\dvol_X(x)
\\&\le e \int_{U^2} K\!\left(s^{-1},x,x\right)\dvol_X(x)
\le Cs^{\n}\ep|U|
\le C'\ep\int_{U}\big(s-V(x)\big)_+^\n\,\dvol_X(x).
\end{aligned}
\end{equation}
The last inequality uses the facts that if $s>2S$ (also note that $S\geq\sup_{x\in U}V(x)$), \be\label{bounded below by half}s\geq s-V(x)\geq s/2.\ee
Similarly, noting \eqref{bounded below by half} and $U\in T(c,\ep_0)$, there exists $C=C(n,c)$, such that
\be\label{integral U is bouned by U1}
\int_{U^2} \big(s- V(x)\big)_+^\n \dvol_X(x)
\le C\ep\int_{U} \big(s- V(x)\big)_+^\n \dvol_X(x).
\ee
Next, we want to apply \Cref{KHL2} to control the integral over $U^1$. 
For this purpose, we need to verify \eqref{exponential bound}, \eqref{uniform limit of quotient}, and
\eqref{error estimate for laplace transform} for
\[
\mu_{x}(s) := e_{H}(s,x,x)
\quad\text{and}\quad
\nu_{x}(s) := (2\pi)^{-n}\omega_n\big(s - V(x)\big)_+^\n.
\]
\paragraph{Verifying \eqref{error estimate for laplace transform}.} For every
$
t < \min\{\big(2S\big)^{-1},1,\ep^{3}\},
$
we obtain
\begin{equation}\label{pointwise weyl law}
\begin{aligned}
&\quad
\Big|
\int_0^\infty e^{-tr}\, d\mu_{x}(r)
-
\int_0^\infty e^{-tr}\, d\nu_{x}(r)
\Big|
=
\left|
K_{H}(t,x,x)
-
(4\pi t)^{-\n} e^{-tV(x)}
\right|
\\
&\le
CSt^{-\n+1}
\le
C'S\, t\,
\Big((4\pi t)^{-\n} e^{-tV(x)}\Big)
=
C'S t
\int_0^\infty e^{-tr}\, d\nu_{x}(r).
\end{aligned}
\end{equation}
Here, for the first inequality, we use \eqref{remainder estimate eq1} in \Cref{heat remainder estimate} .
For the second inequality,  note that,
if
$
t <(2S)^{-1},
$
then
$\label{bounded below by e half}
e^{-tV(x)} \ge e^{-\half}.
$ For the last equality, see \Cref{rem laplace transformation of local integral density}.
\paragraph{Verifying \eqref{exponential bound} and \eqref{uniform limit of quotient}.} For all 
$
t < (2S)^{-1},
$
noting \eqref{bounded below by half}, we have for some $L$ depending only on $n$,
\be\label{verify exponential pointwise}
\int_0^\infty e^{-tr} d\nu_{x}(r)
= (4\pi t)^{-\n} e^{-tV(x)}
\le (4\pi)^{-\n}t^{-\n}\le (2\pi)^{-\n}\big(t^{-1} - V(x)\big)_+^\n
\le L\nu_{x}(t^{-1}).
\ee
Furthermore, by \eqref{bounded below by half}, one can also easily check that whenever $s > 2S$,
\be\label{verify uniform limit pointwise}
\lim_{\tau\to 1} \frac{\nu_{x}(\tau s)}{\nu_{x}(s)} = 1,
\ee
uniformly.
 Fix $\ep>0$, let $\l_1:=2S+1+\ep^{-3}$.
 Let $C_{1}:=C_1(L,\ep)$ be the constant determined in \Cref{KHL2}.
We may apply \Cref{KHL2} to obtain that if $s\geq C_1\l_1$
\begin{equation}\label{pointwise Weyl law on U1}
\left|
e_{H}(s,x,x)
-
\nu_{x}(s)
\right|
\le
(\ep + CSs^{-1})\,
\nu_{x}(s).
\end{equation}
Recall that
$
\nu_{x}(s)
=(2\pi)^{-\n}\omega_n\big(s - V(x)\big)_+^\n,
$ 
by \eqref{pointwise Weyl law on U1}, we have
\begin{equation}\label{Weyl law U1}
\begin{aligned}
&\quad\int_{U^1}
\left| e_{H}(s,x,x)
      - (2\pi)^{-n}\omega_n(s- V(x))_+^\n \right|
\,\dvol_X(x)
\\
&\le
(\ep + C Ss^{-1})(2\pi)^{-n}\omega_n
\int_{U} (s- V(x))_+^\n \dvol_X(x).
\end{aligned}
\end{equation}
The lemma follows immediately from \eqref{Weyl law U2}, \eqref{integral U is bouned by U1} and \eqref{Weyl law U1}.
\end{proof}
\section{Proof of the Main Theorem}\label{proof main}
\subsection{Rescaling of metrics}\label{rescaling}
As discussed in \Cref{idea of local bounded geometry}, we rescale the metric when
estimating $\N(\lambda)$. The criterion $c_\delta(\lambda)\to0$
(i.e.\ \eqref{assumption}) plays two roles. One is that
$c_\delta(\lambda)\to0$ implies $a(\lambda)\to\infty$ in the rescaled sense;
see \eqref{a goes to infty}. The other is to ensure that the rescaled geometric
bounds are uniformly controlled in a neighborhood of $\{ V=\lambda \}$; see
\eqref{rescaled curvature bound} and \eqref{rescaled osc radius}.
This allows us to apply the results in \Cref{local bounded}.

\def\tb{{{r}_\delta(\l)}}
\def\tbs{{\big(\tb\big)^2}}

Consider
\be\label{r delta}
r_\delta(\l):=
    \min\{K_{\delta}(\l)^{-\half},b_\delta(\l)\}.
\ee
Fix $\l\gg1$ and $\delta>0$ sufficiently small temporarily.
 We consider the rescaling of the metric:
\[
g^{\l}:=\big(r_\delta(\l)\big)^{-2} g.
\]
Next, we study how quantities introduced in notation section
\Cref{notation and main result} behave under rescaling.
\def\tN{\tilde{\N}}
\paragraph{Behavior  of $\N$, $\sigma$ and $\Phi$ under rescaling.}
Set $$V^\l:=\big(\tb\big)^2 V \quad\text{and}\quad {H}^\l:=\Delta^\l+V^\l$$ then $H^\l$ is a Schr\"odinger operator on $(M,g^\l)$ (here $\Delta^\l$ is the Laplacian associated to $g^\l$).

Set
\[
\Phi^\l(\mu):=(2 \pi)^{-n}  \omega_{n}\int_M\left(\mu-V^\l(x)\right)^{\n}_+\,\dvol^\l_M(x)\quad\text{and}\quad\sigma^\l(\mu):=\int_{\{V^\l<\mu\}}1\,\dvol^\l_M(x),
\]
where $\dvol^\l_M$ is the volume form induced by $g^\l$.

Set
\be\label{tilde lambda}
\tilde{\l} = \big(\tb\big)^2 \l.
\ee Note that $\Delta^\l=\tbs\Delta$, we have \be H^\l=\tbs H.\ee
Hence we have:
\begin{prop}
Let $\N^\l$ denote the eigenvalue counting function of $H^\l$. 
Then for any $\mu > 0$,
\[
\N(\mu) = \N^\l\!\big(\big(\tb\big)^2 \mu\big),
\qquad
\Phi(\mu) = \Phi^\l\!\big(\tbs \mu\big),
\quad\text{and}\quad
\sigma(\mu) = \sigma^\l\!\big(\tbs \mu\big).
\]
In particular,
we have
\be\label{eigencouting and scaled eigencouting}
\N(\l) = \N^\l(\tilde{\l}),
\ee
and
\be\label{tilde Phi and Phi}
\Phi(\l) = \Phi^\l(\tilde\l)
\quad\text{and}\quad
\sigma(\l) = \sigma^\l(\tilde\l).
\ee
\end{prop}
\paragraph{Behavior of $a$, $b_\delta$, and $d_\delta$ under rescaling.}
Set
\[
a^\l(\tilde{\l}) := \big(\tb\big)^2 a(\l).
\]

\begin{prop}
We have  \[
{a}^\l(\tilde{\l})=\sup\{s\in[0,\infty):2\,\sigma^\l\bigl(\tilde\l-s\bigr)\geq\sigma^\l(\tilde\l+s)\}.
\] 
and \be\label{a goes to infty}
\lim_{\l\to\infty} {a}^\l(\tilde{\l})=\infty.
\ee
\end{prop}
\begin{proof}
  The first equality follows from \eqref{tilde Phi and Phi} and \eqref{defn of a}, and the second follows from \eqref{assumption} and \eqref{r delta}.  
\end{proof}

Let $$\Omega^\l_\mu:=\{V^\l<\mu\}.$$ For $\tilde\mu:=\tbs \mu$, we have
\[
\Omega^\l_{\tilde\mu}=\Omega_\mu.
\]

Set $d^\l_\delta(\tilde\l):=\tbs d_\delta(\l)$ and $b_\delta^\l(\tilde{\l}):=\big(\tb \big)^{-1}b_\delta(\l)$. It is straightforward to see that:
\begin{prop}
We have
\be
d^\l_\delta(\tilde\l)=\sup\Big\{s\in(0,\infty): \l^{\n} \sigma^\l\big(s\big)
\leq
\delta\, \int_M\left(\tilde\l-V^\l(x)\right)^{\n}_+\,\dvol^\l_M(x)\Big\},
\ee
and
\be\label{rescaled b delta}
b_\delta^\l(\tilde{\l})
=
\sup\Bigl\{
r \in [0,\mathrm{inj}_x^\l) :
\osc_{B^\l_r(x)}(V^\l) \leq \delta^2 a^\l(\tilde\l),\ 
x \in \Omega^\l_{\tilde\l+\delta a^\l(\tilde\l)}
\setminus
\Omega^\l_{d^\l_\delta(\tilde\l)}
\Bigr\},
\ee
where $B^\l_r(x)$ denotes the geodesic ball of radius $r$ centered at $x$
with respect to $g^\l$, and $\mathrm{inj}_x^\l$ is the injectivity radius
at $x$ with respect to $g^\l$.
In particular, for any
$x \in \Omega^\l_{\tilde\l+\delta a^\l(\tilde\l)}
\setminus
\Omega^\l_{d^\l_\delta(\tilde\l)}$,
\be\label{rescaled osc radius}
\mathrm{inj}_x^\l \ge b_\delta^\l(\tilde{\l}) \ge 1.
\ee
\end{prop}

\paragraph{Behavior of curvature bounds under rescaling.}
Consider
\[
R^\l_\delta(\tilde\l) := \tbs R_\delta(\l),
\qquad
S^\l_\delta(\tilde\l) := \big(\tb\big)^3 S_\delta(\l),
\qquad
T^\l_\delta(\tilde\l) := \big(\tb\big)^4 T_\delta(\l) \ge 0.
\]
It is straightforward to see that
\begin{prop}
The norm of the curvature operator, as well as the norms of its first and second
covariant derivatives, on the $1$-neighborhood (with respect to $g^\l$) of
\[
\Omega^\l_{\tilde\l + \delta a^\l(\tilde\l)}
\setminus
\Omega^\l_{d^\l_\delta(\tilde\l)}
\]
are bounded above by $R^\l_\delta(\tilde\l)$, $S^\l_\delta(\tilde\l)$, and
$T^\l_\delta(\tilde\l)$, respectively. Here for $r > 0$, by the \emph{$r$-neighborhood} of a set $A \subset M$
(with respect to $g^\l$) we mean the open set
\[
\cup_{x \in A} \mathring{B}^\l_r(x).
\]
Moreover, we have
\be\label{rescaled curvature bound}
R^\l_\delta(\tilde\l),{S^\l_\delta(\tilde\l), T^\l_\delta(\tilde\l)} \le 1.
\ee

\end{prop}

\subsection{ Voronoi-type tessellation in rescaled space}\label{Vor}
We use the same notation as in \Cref{rescaling}.
Let
\[
O_\lambda
:=
\Omega^\lambda_{\tilde\lambda+\delta a^\lambda(\tilde\lambda)}
\setminus
\Omega^\lambda_{d_\delta^\lambda(\tilde\lambda)}
=
\Omega_{\lambda+\delta a(\lambda)}
\setminus
\Omega_{d_\delta(\lambda)}.
\]

Applying the construction in \Cref{vitali} with
$W=O_\lambda$, $(X,g)=(M,g^\lambda)$, $r_0=8^{-1}$, and $R=1$
(noting \eqref{rescaled curvature bound} and
\eqref{rescaled osc radius}),
we obtain a ``tessellation''
\[
D_0 \cup \big(\cup_{j=1}^{J} Q_j\big)
\]
of $M$ satisfying the properties listed in
\Cref{properties of tessellation0}.

It follows from item~(3) in \Cref{properties of tessellation0} and the intermediate
value theorem for continuous functions that each connected component of $D_0$
is either contained in $\Omega_{d_\delta(\lambda)}(=\{V<d_\delta(\lambda)\})$ or in
$
M \setminus \Omega_{\lambda+\delta a(\lambda)}
=
\{V\ge \lambda+\delta a(\lambda)\}.
$
Accordingly, we further decompose $D_0$ as
\[
Q_0 := D_0 \cap \{V<d_\delta(\lambda)\},
\qquad
Q_{J+1} := D_0 \cap \{V\ge \lambda+\delta a(\lambda)\}.
\]

Then we have:

\def\tQ{\tilde{Q}}

\begin{prop}\label{heat kernle estimate}
For $1\le j\le J$, \be\label{heat kernel estimate on Qj} Q_j\in G(c_1)\ee for some constants $c_1$ depending only on $n$ (associated with  $(M,g^\l)$).

Assume that $\l$ is sufficiently large.  Let $K_0(t,x,y)$ denote either the Dirichlet or Neumann heat kernel on $Q_0$ (with respect to the Beltrami-Laplacian associated with $g^\lambda$). For $t\in(0,\tilde{\lambda}^{-1}]$, there exists $(\lambda,\delta)$-independent constant $c_4$ such that on $Q_0$,
\begin{equation}\label{heat kernel estimate on Q0}
0\;\le\; K_0(t,x,x)
\;\le\; c_4\, t^{-\n}.
\end{equation}
\end{prop}

To prove \Cref{heat kernle estimate}, we first introduce:
\begin{defn}
   We say $(M,g)$ is $(C,c)$-{tamed by $V\geq0$} if there exist constants $c, C > 0$ such that 
   \[
   |R_g|(x) \leq C\bigl(V(x)+1\bigr) \quad \text{and} \quad 
   \mathrm{inj}_x \geq c\bigl(V(x)+1\bigr)^{-\half},
   \]
   where $R_g$ denotes the curvature operator with respect to $g$.

\end{defn}
Then we have:
\begin{lem}\label{tameness}
$(M,g)$ is tamed by $V$ for some $(\lambda,\delta)$-independent constants $(C,c)$.
\end{lem}

\begin{proof}
Note that $c_\delta(\lambda)\to0$ (i.e.\ \eqref{assumption}) holds for
sufficiently small $\delta$; fix such a $\delta_0>0$ and assume that
$c_{\delta_0}(\lambda)\to0$.
Then $c_{\delta_0}(\lambda)\to0$ implies
$
a(\lambda)\, b_{\delta_0}^2(\lambda)\to\infty,
$ and $
\frac{R_{\delta_0}(\lambda)}{a(\lambda)}\to0 .
$
Since $a(\lambda)<\lambda$, it follows that
$
\lambda\, b_{\delta_0}^2(\lambda)\to\infty,
$ and $
\frac{R_{\delta_0}(\lambda)}{\lambda}\to0.
$
In particular, for $\lambda$ sufficiently large, the curvature on the
level set $\{V=\lambda\}$ is bounded above by $\lambda$, and the
injectivity radius at every point of $\{V=\lambda\}$ is bounded below
by $\lambda^{-\half}$. This finishes the proof.
\end{proof}

\begin{proof}[Proof of \Cref{heat kernle estimate}]
The case $1\le j\le J$ is nothing but item (5) in \Cref{tesslation} (Note that,
under the rescaled metric, the bounds \eqref{rescaled curvature bound} and
\eqref{rescaled osc radius} hold). For $j=0$, we consider the rescaled metric
\[
\tilde g^\lambda := \tilde\lambda\, g^\lambda = \lambda\, g .
\]
Note that, by \Cref{tameness}, $(M,g)$ is tamed by $V$ for some
$(\lambda,\delta)$-independent $(C,c)$. Hence, under this rescaled metric,
since
$
Q_0 \subset \Omega_{d_\delta(\lambda)} \subset \Omega_\lambda,
$
we have $Q_0 \in \WBG(2C,c/2)$.

Moreover, by item~(3) in \Cref{properties of tessellation0},
$Q_0 \in \Cone(L,\rho,\varepsilon_0,\theta)$ with respect to the metric
$g^\lambda$, where the constants $(L,\rho,\varepsilon_0,\theta)$ depend
only on $n$. Since $\tilde\lambda \ge a^\lambda(\tilde\lambda)$ and
$a^\lambda(\tilde\lambda)\to\infty$ (see \eqref{a goes to infty}),
we have $\tilde\lambda>1$ for $\lambda$ sufficiently large. It follows
that $Q_0 \in \Cone(L,\rho,\varepsilon_0,\theta)$ also with respect to
the metric $\tilde g^\lambda$.

The estimate \eqref{heat kernel estimate on Q0} then follows from
\Cref{Cone implies SObolev} and \Cref{Varo}.

\end{proof}

\def\unlessstate{1}
\if\unlessstate0
Unless stated otherwise, all constants are independent of $(j,\lambda,\delta)$.

Let $1\le j\le J$. By \eqref{rescaled curvature bound}, the (inverse) exponential map
\[
\exp^\lambda_{x_j}:T_{x_j}M\to M
\]
identifies $Q_j$ with a Lipschitz domain $\widetilde Q_j\subset T_{x_j}M\cong\R^n$.  
For $x\in Q_j$, write $\tilde{x}=(\exp^\lambda_{x_j})^{-1}(x)\in\widetilde Q_j$.  
For $r>0$ and $B^{Q_j}_r(x)=\{y\in Q_j:d^\lambda(y,x)<r\}$, set
\[
B^{\widetilde Q_j}_r(\tilde{x})
:= (\exp^\lambda_{x_j})^{-1}(B^{Q_j}_r(x)).
\]

By the Rauch Jacobi Field Comparison Theorem and Toponogov’s comparison theorem, there exist $C>1$ and $\varepsilon_0>0$ such that
\begin{equation}\label{Lipschitz}
C^{-1}|\tilde x-\tilde y|
\le d^\lambda(\exp^\lambda_{x_j}(\tilde x),\exp^\lambda_{x_j}(\tilde y))
\le C|\tilde x-\tilde y|,
\qquad \tilde x,\tilde y\in\widetilde Q_j.
\end{equation}

By \cite[Theorem~1]{boulkhemair2007uniform} and item \eqref{tess item 3} in \Cref{properties of tessellation}, we obtain a uniform weak Poincaré inequality: there exists $C_P'>0$ such that for any $\tilde{x}\in\widetilde Q_j$, any $r<\varepsilon_0/2$, and any $u\in C^\infty(B^{\widetilde Q_j}_r(\tilde{x}))$,
\begin{equation}\label{Poincare0}
\int_{B^{\widetilde Q_j}_r(\tilde{x})}
\Big|u(z)-|B^{\widetilde Q_j}_r(\tilde{x})|^{-1}\!\!\int_{B^{\widetilde Q_j}_r(\tilde{x})}u(y)\,dy\Big|^2 dz
\le
C_P' r^2 \int_{B^{\widetilde Q_j}_{2r}(\tilde{x})} |\nabla^{\R^n} u|^2(z)\,dz,
\end{equation}
where $\nabla^{\R^n}$ is the gradient in Euclidean space $\R^n$.

Using \eqref{Lipschitz} and \eqref{Poincare0}, we obtain constants $C_P>0$ and $\varepsilon_0'>0$ such that for any $x\in Q_j$, any $r<\varepsilon_0'/2$, and any $u\in C^\infty(B^{Q_j}_r(x))$,
\begin{equation}\label{Poincare1}
\int_{B^{Q_j}_r(x)}
\Big|u(z)-|B^{Q_j}_r(x)|^{-1}\!\!\int_{B^{Q_j}_r(x)}u(y)\,dy\Big|^2 dz
\le
C_P r^2 \int_{B^{Q_j}_{2r}(x)} |\nabla^\l u|^2(z)\,dz,
\end{equation}
where $\nabla^\l$ is the gradient induced by $g^\l.$

The $\varepsilon_0$–cone property also implies the volume doubling estimate: there exists $C_D>1$ such that
\begin{equation}\label{doubling}
|B^{Q_j}_{2r}(x)| \le C_D |B^{Q_j}_r(x)|.
\end{equation}

The estimate \eqref{heat kernel estimate on Qj} for $1\le j\le J$ now follows from \eqref{Poincare1}, \eqref{doubling}, and Sturm’s results \cite[Theorem~2.6]{sturm1996analysis} and \cite[Theorem~2.4]{sturm1995analysis}.

For $j=0$, we apply the same argument to the rescaled metric $\tilde{g}^\lambda := \tilde{\lambda} g^\lambda$, which yields \eqref{heat kernel estimate on Q0}.\fi

\def\vitali{1}
\if\vitali0

Fix such a covering. Set $U^\l:=M\setminus\overline{O_\l}$.  
Let $\{\phi_j\}_{j=0}^J$ be a partition of unity subordinated to the open cover
\[
U_\l\ \cup\ \Bigl(\,\bigcup_{j=1}^J B^\l_{8^{-1}}(x_j)\Bigr)
\]
of $M$, where $\supp(\phi_0)\subset U_\l$ and $\supp(\phi_j)\subset B^\l_{8^{-1}}(x_j)$ for $j\ge1$.  
Moreover, there exist constants $C_k$, depending only on $k$, such that
\be\label{Gradient upper bound of phi}
|(\nabla^{\l})^k \phi_j|\le C_k,
\ee
where $\nabla^\l$ denotes the gradient with respect to $g^\l$.

Set
\[
V_j^\l:=\inf_{x\in B^\l_{8^{-1}}(x_j)} V^\l(x), 
\qquad\text{and}\qquad
\tvl:=\phi_0 V+\sum_{j=1}^J \phi_j V_j.
\]
It is straightforward to check the following.

\begin{lem}

On $M$, by \eqref{rescaled osc radius} we have
\be\label{Relation btw Vl and tvl}
\tvl \;\le\; V^\l \;\le\; \tvl + \delta^2 a^\l(\tilde\l).
\ee

The function $\tvl$ is smooth on the $\tfrac18$-neighborhood (with respect to $g^\l$) of
\[
\Omega^\l_{\tilde\l+\delta a^\l(\tilde\l)}\setminus\Omega^\l_{d_\delta^\l(\tilde\l)},
\]
denoted by $O_\l^1$.

Moreover, on $O_\l^1$ there exist constants $C_k$, depending only on $k$, such that
\be\label{Gradient upper bound of tvl}
|(\nabla^{\l})^k \tvl|\le C_k\,\delta^2\,\min\{a^\l(\tilde\l),\tvl(x)\}.
\ee
\end{lem}

\begin{proof}
The smoothness and estimate \eqref{Relation btw Vl and tvl} are immediate. It remains to prove \eqref{Gradient upper bound of tvl}.

For any $x\in O_\l^1$, let
\[
J_x:=\big\{\, j\in\{1,\dots,J\}:\ x\in B^\l_{8^{-1}}(x_j) \,\big\}.
\]
Fix $j_0\in J_x$. Then
\[
\begin{aligned}
(\nabla^\l)^k \tvl(x)
&= \sum_{j\in J_x} V_j\, (\nabla^\l)^k \phi_j \\
&= \sum_{j\in J_x} V_{j_0}\, (\nabla^\l)^k \phi_j 
   + \sum_{j\in J_x} (V_j - V_{j_0})\, (\nabla^\l)^k \phi_j \\
&= \sum_{j\in J_x} (V_j - V_{j_0})\, (\nabla^\l)^k \phi_j,
\end{aligned}
\]
since $\sum_{j\in J_x} \phi_{j}(x) \equiv 1$ implies $\sum_{j\in J_x} (\nabla^\l)^k \phi_j(x) = 0$.

Now the estimate follows from the identity above, together with  
\eqref{Gradient upper bound of phi} and \eqref{rescaled osc radius}.
\end{proof}

\subsubsection{Garding's inequality for Schr\"odinger operators}

\fi

\def\sq{{\mathrm{sq}}}
\def\before{1}
\if\before0
\subsection{Finite propogation speed}
Before proceeding, the following lemma is needed.
\begin{lem}\label{Fourier}
   Let $N_2>N_1+1>1$ be two fixed constants. Let $\rho_{N_1}^{N_2}\in C^\infty(\R)$ such that
    \begin{enumerate}[(1)]
       \item $\rho_{N_1}^{N_2}|_{(-\infty,N_1]}\equiv N_1$, $\rho_{N_1}^{N_2}(t)=t$ if $|t|\geq N_2$.
       \item $|(\rho_{N_1}^{N_2})'(t)|\leq 2$ and $|(\rho_{N_1}^{N_2})''(t)|\leq 16(N_2-N_1)^{-1}$.
    \end{enumerate}
    
    For any $t>0$, consider the following function $E_{N_1,t}^{N_2}\in C(\R)$
    \[
   E_{N_1,t}^{N_2}(x):=e^{-t\rho_{N_1}^{N_2}(x^2)+tN_1} 
    \]
      Let $$F_{N_1,t}^{N_2}(\xi):=\int_{\R}e^{-i x\cdot\xi}E_{N_1,t}^{N_2}(x)dx$$ be the Fourier transformation of $E_{N_1}^{N_2}$, then  there exists some universal constant $C$ such that
      \[
      |F_{N_1,t}^{N_2}(\xi)|\leq C(|\xi|^2+1)^{-1}.
      \]
\end{lem}
\begin{proof}
    We compute
    \[\ba
    &\quad\int_{\R} e^{-i x\cdot\xi}E_{\Lambda,t}(x)dx=\int_{-\sqrt{\Lambda}}^{\sqrt{\Lambda}} e^{-ix\cdot\xi}dx+\int_{\sqrt{\Lambda}}^{\infty} e^{-ix\cdot\xi} e^{-t(x^2-\Lambda)}dx+\int^{-\sqrt{\Lambda}}_{-\infty} e^{-ix\cdot\xi} e^{-t(x^2-\Lambda)}dx\\
    &=\frac{i}{\xi}(e^{-i\sqrt\Lambda\cdot\xi}-e^{i\sqrt\Lambda\cdot\xi})+\int_{\sqrt{\Lambda}}^{\infty} e^{-ix\cdot\xi} e^{-t(x^2-\Lambda)}dx+\int^{-\sqrt{\Lambda}}_{-\infty} e^{-ix\cdot\xi} e^{-t(x^2-\Lambda)}dx\\
    &=\frac{-2i}{\xi}\int_{\sqrt{\Lambda}}^\infty tx e^{-ix\cdot\xi} e^{-t(x^2-\Lambda)} dx+\frac{-2i}{\xi}\int^{-\sqrt{\Lambda}}_{-\infty} tx e^{-ix\cdot\xi} e^{-t(x^2-\Lambda)} dx.
    \ea\]
    Note that
    \[
 \left|2\int_{\sqrt{\Lambda}}^\infty tx e^{-ix\cdot\xi} e^{-t(x^2-\Lambda)} dx\right|\leq 2\int_{\sqrt{\Lambda}}^\infty tx e^{-t(x^2-\Lambda)} dx=1
    \]
    and
  \[
 \left|2\int^{-\sqrt{\Lambda}}_{-\infty} tx e^{-ix\cdot\xi} e^{-t(x^2-\Lambda)} dx\right|\leq 2\int_{\sqrt{\Lambda}}^\infty tx e^{-t(x^2-\Lambda)} dx=1,
    \]
    the lemma follows.
\end{proof}

\begin{lem}\label{Fourier2}
We use the same notation as in \Cref{Fourier}.
       Let $\phi\in C_c^\infty(-1,1)\subset C_c^\infty(\R)$ be a bump function such that $\phi(\xi)\equiv 1$ whenever $|xi|\leq1$.

      Let 
      \[E^1_{\Lambda,t}(x):=\frac{1}{2\pi}\int_\R e^{ix\cdot\xi}\phi(\xi)F_{\Lambda,t}(\xi)d\xi\]
      and
      \[E^2\]
\end{lem}


Now we use the same notation as in \cref{rescaling} and \Cref{Fourier}.

Let
\[
O^\l:=\Omega^\l_{\tilde\l+\delta a^\l(\tilde\l)}\setminus\Omega^\l_{d_\delta^\l(\tilde\l)}.
\]

Fix a $x_0\in O^\l$, let $B^\l_r:=B_r^\l(x_0)$, i.e., the geodesci ball of raduius $r$ w.r.t. $g^\l$, centering at $x_0$.

Let $H^\l_D$ be restriction of $H^\l$ on $B_1^\l$ with Dirichlet boundary condition.
Let $\Lambda:=\inf_{x\in B^\l}V^\l$, then since $H^\l_D\geq \Lambda$, one can see easily that
\be\label{Dirichlet kernel is heat kernel}
E_{\Lambda,t}\left(\sqrt{H^\l_D}\right)=e^{-t(H^\l_D-\Lambda)}.
\ee

\def\kl{{K_\Lambda^\l}}
\def\kld{{K_{\Lambda,D}^\l}}

Let $\kl(t,x,y)$ and $K_{\Lambda,D}^\l(t,x,y)$ be the integral kernel of $E_{\Lambda,t}\left(\sqrt{H^\l}\right)$ and $E_{\Lambda,t}\left(\sqrt{H_D^\l}\right)$ respectively. Then by \eqref{Dirichlet kernel is heat kernel}, the kernel $K_{\Lambda,D}^\l$ is nothing but the heat kernel of the positive operator $H_D^\l-\Lambda$.

\begin{lem}
For any $\delta>0$, there exists $\l_\delta>0$, such that when $\l>\l_\delta$, we have
for every $x\in\Omega_{\l+\delta a(\l)},$
\[|e_H(\l,x,x)-e_{H^{x}}(\l,x,x)|\leq \delta e_{H^x}(\l,x,x).\]
\end{lem}
\begin{proof}
    
\end{proof}
\fi
\subsection{DN bracketing method}\label{proof of thm 13}

We are now in a position to apply DN bracketing method to prove \Cref{main}. We continue to use the notation introduced in \cref{rescaling} and \cref{Vor}.
\subsubsection{Reducing to the region
$\Omega_{\l-\delta a(\l)} \setminus \Omega_{d_\delta(\l)}$}\label{reduction to near V=lambda}

Using the DN bracketing argument, we will show in this subsection that,
up to an error of order $\approx \delta \Phi(\l)$, it suffices to focus the discussion on
the region $\Omega_{\l-\delta a(\l)} \setminus \Omega_{d_\delta(\l)}$.

For each $Q_j$, let $H^\lambda_{Q_j,D}$ and $H^\lambda_{Q_j,N}$ denote the restriction of $H^\lambda$ to $Q_j$ with Dirichlet and Neumann boundary conditions, respectively. Let $\mathcal{N}^\lambda(\mu,H^\lambda_{Q_j,D/N})$ be the corresponding eigenvalue counting function. By domain monotonicity, we have
\be\label{domain monotonicity}
\sum_{j=0}^{J+1} \mathcal{N}^\lambda(\tilde\lambda,H^\lambda_{Q_j,D})
\;\le\;
\mathcal{N}^\lambda(\tilde\lambda)
\;\le\;
\sum_{j=0}^{J+1} \mathcal{N}^\lambda(\tilde\lambda,H^\lambda_{Q_j,N}).
\ee
\paragraph{On $Q_0$ and $Q_{J+1}$.} 
We have
\be\label{Weyl law Q0}
\begin{aligned}
&\quad\mathcal{N}^\l(\tilde\l, H^\l_{Q_0,D})
\le
\mathcal{N}^\l(\tilde\l, H^\l_{Q_0,N})
\le
\mathcal{N}^\l(\tilde\l, \Delta^\l_{Q_0,N})
\\
&\le
C_3 \tilde\l^\n |Q_0|^\l
\le
C_3 \l^\n \bigl| \Omega_{d_\delta(\l)} \bigr|
\leq
C_3\, \delta\, \Phi(\l),
\end{aligned}
\ee
where $|Q_0|^\l$ denotes the volume of $Q_0$ with respect to $g^\l$, and
$\Delta^\l_{Q_0,N}$ denotes the restriction of $\Delta^\l$ to $Q_0$ with
Neumann boundary condition.
Here the third inequality follows from the heat kernel estimate
\eqref{heat kernel estimate on Q0} and
\Cref{lem:spectral-bound-from-heat-kernel}, the fourth inequality follows
from the relation between the rescaled volume and the original volume,
together with  $Q_0 \subset \Omega_{d_\delta(\l)}$, and the
last inequality follows from the definition of $d_\delta(\l)$, or
\eqref{spectral on d delta is small}.

 It is also immediate that (since $V^\l\geq\tilde{\l}$ on $Q_{J+1}$)
\be\label{Weyl law QJ}
\mathcal{N}^\lambda(\tilde\lambda,H^\lambda_{Q_{J+1},D})
=
\mathcal{N}^\lambda(\tilde\lambda,H^\lambda_{Q_{J+1},N})
=0.
\ee
\paragraph{On $Q_j,j\geq1$.} It remains to estimate
\[
\mathcal{N}^\lambda(\tilde\lambda,H^\lambda_{Q_{j},D/N}),
\qquad 1\le j\le J.
\]
We may as well assume that there exists $J'<J$ such that   
\[
Q_j\cap\Omega^\l_{\tilde\l-\delta a^\l(\tilde\l)}\neq\emptyset,
\]
iff  $1\le j\le J'$.

Let $C_j:=\inf_{x\in Q_j}V^\l(x)$ and 
\[
\tH^\l_{Q_j,D/N}:=H^\l_{Q_j,D/N}-C_j.
\]
Then we clearly have
$
\N^\l(\tilde\l,H^\l_{Q_j,N})
=
\N^\l(\tilde\l-C_j,\tH^\l_{Q_j,N}).
$
Thus, using \eqref{spectrum between l-a(l) an l+a(l)},
\Cref{lem:spectral-bound-from-heat-kernel}, and
\eqref{heat kernel estimate on Qj} the same way as in \eqref{Weyl law Q0}, we obtain if $\delta\in\big(0,\half\big)$
\begin{equation}\label{Weyl law Qj, j big}
\begin{aligned}
0
&\le
\sum_{j=J'+1}^J \N^\l\big(\tilde\l, H^\l_{Q_j,D}\big)
\le
\sum_{j=J'+1}^J \N^\l\big(\tilde\l, H^\l_{Q_j,N}\big)
=
\sum_{j=J'+1}^J \N^\l\big(\tilde\l - C_j, \tH^\l_{Q_j,N}\big)
\\
&\le
\sum_{j=J'+1}^J \N^\l\big(3\delta a^\l(\tilde\l), \Delta^\l_{Q_j,N}\big)
\le
C_4 \big(3\delta a^\l(\tilde\l)\big)^\n
\sum_{j=J'+1}^J |Q_j|^\l
\\
&\le
C_4 \big(3\delta a(\l)\big)^\n
\bigl|
\Omega_{\l+2\delta a(\l)} \setminus \Omega_{\l-2\delta a(\l)}
\bigr|
\le
C_4'\, \delta^\n\, \Phi(\l),
\end{aligned}
\end{equation}
where $\Delta^\l_{Q_j,N}$ denotes restriction of $\Delta^\l$ on $Q_j$ with Neumann boundary condition.

\subsubsection{Dealing with $Q_j,1\leq j\leq J'$}\label{dealing}
Now, it suffices to focus the discussion on
$Q_j$, $1 \le j \le J'$, which lie inside the $1$-neighborhood (with respect to $g^\l$) of the region
$\Omega_{\l-\delta a(\l)} \setminus \Omega_{d_\delta(\l)}$.
For this purpose, we apply \Cref{quantative Weyl law} to  $\tilde{H}^\l_{Q_j,D/N}(={H}^\l_{Q_j,D/N}-C_j)$, to obtain:
\begin{lem}
For each $\ep\in(0,1)$, there exist constants
$C=C(\ep,n), C'=C'(n)>0$, $\delta_0=\delta_0(\ep,n)$ and  such that whenever $\delta<\delta_0$ and
$\l\ge\l_0$ for some $\l_0=\l_0(\delta,\ep)>0$, for each $1\le j\le J'$,
\begin{equation}\label{Weyl law Qj, j small}
\begin{aligned}
\N^\l(\tilde\l,H^\l_{Q_j,N})
&\le
\Big(1+C'\ep+C\delta\Big)(2\pi)^{-n}\omega_n
\int_{Q_j} (\tilde\l-V^\l(x))_+^\n\,\dvol_M^\l(x),
\\
\N^\l(\tilde\l,H^\l_{Q_j,D})
&\ge
\Big(1-C'\ep-C\delta\Big)(2\pi)^{-n}\omega_n
\int_{Q_j} (\tilde\l-V^\l(x))_+^\n\,\dvol_M^\l(x).
\end{aligned}
\end{equation}
\end{lem}

\begin{proof}
Fix $j\in\{1,2,\dots,J'\}$. Recall that
$\tH^\l_{Q_j,D/N}:=H^\l_{Q_j,D/N}-C_j$. Then
\be\label{shifted eigenvalue counting}
\N^\l(\tilde\l,H^\l_{Q_j,D/N})
=
\N^\l(\tilde\l-C_j,\tH^\l_{Q_j,D/N}).
\ee
We will apply \Cref{quantative Weyl law} with
$H=\tH^\l_{Q_j,D/N}$.

By the rescaled geometric bounds
\eqref{rescaled curvature bound},
\eqref{rescaled osc radius}, and
\Cref{properties of tessellation0}, we have
\[
Q_j\in \BG(1,1)\cap G(c_1(n))\cap T(c(n),\ep_0(n)).
\]
Let $$s=\tilde\l-C_j.$$ By assumption, for $1\le j\le J'$,
$
Q_j\cap \Omega^\l_{\tilde\l-\delta a^\l(\tilde\l)}\neq\emptyset.
$
Recall $C_j:=\inf_{x\in Q_j}V^\l$,
\be\label{lower r}
s\ge \delta\, a^\l(\tilde\l).
\ee
The potential of $\tH^\l_{Q_j,D/N}$ is $V^\l-C_j$, which satisfies (note \eqref{rescaled b delta})
\be\label{upperbound of VQj}
0\le \sup_{x\in Q_j}\big(V^\l(x)-C_j\big)
\le \delta^2 a^\l(\tilde\l)=:S.
\ee
Let $\delta_0:=\Big(2\cc\big(\ep,n\big)\Big)^{-1}$,  where $\cc$ is the constant appearing in \Cref{quantative Weyl law}.  By \eqref{a goes to infty}, \eqref{lower r}, and
\eqref{upperbound of VQj}, there exists $\l_0=\l_0(\ep,\delta)$ such that,
whenever $\l>\l_0$,
\be
s>2\cc^2
\qquad\text{and}\qquad
S\ge 1,
\ee
 In
particular, for $\l\ge\l_0,\delta<\delta_0$ we have $s>\cc(2S+\cc)$. Noting that
$Ss^{-1}\leq\delta$ and \eqref{shifted eigenvalue counting}, the conclusion follows from
\Cref{quantative Weyl law}.
\end{proof}

\begin{lem}There exists $C=C(n)>0$, such that
\be\label{approach Phi lambda}
   \bigg|\sum_{j=1}^{J'}(2\pi)^{-n}\omega_n
\int_{Q_j} (\tilde\l-V^\l(x))_+^\n\,\dvol_M^\l(x)-\Phi(\l)\bigg|\leq (C\delta+C\delta^\n)\Phi(\l).\ee
\end{lem}
\begin{proof}All constants appearing in the proof depend only on $n$.
Note that
\be\label{decompose integral}
\Phi(\l)=\sum_{j=0}^{J+1}(2\pi)^{-n}\omega_n
\int_{Q_j} (\l-V(x))_+^\n\,\dvol_M(x).
\ee
Since $Q_0\subset\Omega_{d_\delta(\l)}$, by \eqref{spectral on d delta is small},
\be\label{integral density on Q0}
\int_{Q_0} (\l-V(x))_+^\n\,\dvol_M(x)
\le \l^\n|Q_0|
\le \l^\n|\Omega_{d_\delta(\l)}|
\le C\delta\Phi(\l).
\ee
It is straightforward to see that
\be\label{integral density on Q J+1}
\int_{Q_{J+1}} (\l-V(x))_+^\n\,\dvol_M(x)=0.
\ee
Similar to the proof of \eqref{Weyl law Qj, j big}, by
\eqref{spectrum between l-a(l) an l+a(l)},
\be\label{integral density on Qj, j large}
\sum_{j=J'+1}^J\int_{Q_j} (\l-V(x))_+^\n\,\dvol_M(x)
\le C\big(\delta a(\l)\big)^\n
|\Omega_{\l+2\delta a(\l)}\setminus\Omega_{\l-2\delta a(\l)}|
\le C'\delta^\n\Phi(\l).
\ee
The lemma then follows from
\eqref{decompose integral}--\eqref{integral density on Qj, j large}
and the fact that
\[
\int_{Q_j}(\tilde\l-V^\l(x))_+^\n\,\dvol_M^\l(x)
=
\int_{Q_j}(\l-V(x))_+^\n\,\dvol_M(x).
\]
\end{proof}
By \eqref{Weyl law Qj, j small} and \eqref{approach Phi lambda}, we have
\begin{cor}
    For each $\ep\in(0,1)$, there exist constants
$C=C(\ep,n), C'=C'(n)>0, C''=C''(n)>0$, and $\delta_0=\delta_0(\ep,n)$ such that whenever $\delta<\delta_0$ and
$\l\ge\l_0$ for some $\l_0=\l_0(\delta,\ep)>0$, 
\begin{equation}\label{sum Weyl law Qj, j small}
\begin{aligned}
\sum_{j=1}^{J'}\N^\l(\tilde\l,H^\l_{Q_j,N})
&\le
\Big(1+C'\ep+C\delta\Big)\Big(1+C''\delta+C''\delta^\n\Big)\Phi(\l),
\\
\sum_{j=1}^{J'}\N^\l(\tilde\l,H^\l_{Q_j,D})
&\ge
\Big(1-C'\ep-C\delta\Big)\Big(1-C''\delta-C''\delta^\n\Big)\Phi(\l).
\end{aligned}
\end{equation}
\end{cor}

Now by \eqref{domain monotonicity}, \eqref{Weyl law Q0}, \eqref{Weyl law QJ}, \eqref{Weyl law Qj, j big}, and \eqref{sum Weyl law Qj, j small}, for any $\ep > 0$, there exists $\delta=\delta(\ep,n) > 0$ small enough such that there exists $\l_0 = \l_0(\delta,\ep)$, and whenever $\l > \l_0$,
\[
\big| \N^\l(\tilde\l) - \Phi^\l(\tilde\l) \big| \leq \ep \, \Phi(\l).
\]
By \eqref{eigencouting and scaled eigencouting}, this is equivalent to
\[
\left| \N(\l) - \Phi(\l) \right| \leq \ep \, \Phi(\l),\quad\l\geq\l_0.
\]
This concludes the proof of \Cref{main}.

\def\subsec{1}
\if\subsec0
\subsection{Pointwise Weyl law}
\begin{thm}
For a.e.\ $x$ satisfying $A<V(x)<\lambda-\sqrt\delta\, \da(\lambda)/8$ for some large $A>0$, there exists a constant $C>0$, independent of $\delta$ and $\lambda$, such that
\begin{equation}
\big(1- C(\delta + c_\delta(V(x)))\big)\, (\lambda - V(x))^{\n} 
\;\le\; e(\l,x,x)\;\le\; 
\big(1+ C(\delta + c_\delta(V(x)))\big)\, (\lambda - V(x))^{\n}.
\end{equation}
\end{thm}
\begin{proof}

 Let $K_0(t,x,y)=(4\pi t)^{-\n}e^{-\frac{d^2(x,y)}{4t}}e^{-tV(x)}$
\end{proof}
\fi

\section{Sharpness of our conditions}\label{hyperbolic example}

\subsection{The example on $\R\times S^1$}\label{R times s1}

On $\R$, let $V$ be a smooth function satisfying
\[
V(x)=\sqrt{\ln(|x|)} \qquad (|x|\gg1).
\]
By our main result 
\Cref{main} (or \cite[Theorem~4.1]{tachizawa1992eigenvalue}), the one–dimensional operator
\[
H^{\R}=-\p_x^2+V(x)
\]
obeys the Weyl asymptotic
\begin{equation}\label{eq simple1}
\N(\lambda,H^{\R})
\sim \pi^{-1}\!\int_{\R}(\lambda-V(x))_{+}^{1/2}dx,
\qquad \lambda\to\infty,
\end{equation}
where $\N(\lambda,H^{\R})$ denotes the eigenvalue counting function of $H^{\R}$.

A direct computation gives
\begin{equation}\label{eq simple 10}
\begin{aligned}
&\quad\int_{1}^{e^{\lambda^{2}}}\!\bigl(\lambda-\sqrt{\ln x}\bigr)^{1/2}dx
=2\int_{0}^{\lambda} (\lambda-r)^{1/2}r e^{r^{2}}dr  \\
&=2e^{\lambda^{2}}\int_{0}^{\lambda} s^{1/2}(\lambda-s)e^{-2\lambda s+s^{2}}ds \ge e^{\lambda^{2}}\int_{0}^{\lambda/2} s^{1/2}\lambda e^{-2\lambda s}ds\\
&= e^{\lambda^{2}}\l^{-\half}\int_{0}^{\lambda^2/2} u^{1/2}e^{-2u}du \;\gtrsim\;e^{\lambda^{2}}\lambda^{-1/2}.
\end{aligned}
\end{equation}
Let $S^{1}$ be the unit circle, $\Delta^{S^{1}}$ its Laplacian, and consider the 
product operator on $\R\times S^{1}$:
\[
H^{\R\times S^{1}} := H^{\R} + \Delta^{S^{1}}.
\]
Applying the same change of variables as in \eqref{eq simple 10} yields
\begin{equation}\label{eq simple2}
\begin{aligned}
&\quad\int_{1}^{e^{\lambda^{2}}}\!\bigl(\lambda-\sqrt{\ln x}\bigr)\,dx
=2e^{\lambda^{2}}\int_{0}^{\lambda} s(\lambda-s)e^{-2\lambda s+s^{2}}ds \\
&\le e^{\lambda^{2}}\int_{0}^{\lambda} s\lambda e^{-\lambda s}ds=e^{\lambda^{2}}\l^{-1}\int_{0}^{\lambda^2} u e^{-u}du
 \;\lesssim\; \,e^{\lambda^{2}}\lambda^{-1}.
\end{aligned}
\end{equation}
Meanwhile, note that $\Delta^{S^1}$ has eigenvalue $0$. Consequently,
\begin{equation}\label{eq simple3}
\N(\lambda,H^{\R\times S^{1}})\ge \N(\lambda,H^{\R}),
\end{equation}
where $\N(\lambda,H^{\R\times S^1})$ denotes the eigenvalue counting function of $H^{\R\times S^1}$.

Combining \eqref{eq simple1}–\eqref{eq simple3}, we conclude that
\be\label{failure}
\int_{S^{1}}\!\int_{\R}(\lambda-V(x))_{+}\,dx\,d\theta
= o\!\left(\N(\lambda,H^{\R\times S^{1}})\right),
\qquad \lambda\to\infty,
\ee
and therefore the classical Weyl law fails for the product operator 
$H^{\R\times S^{1}}$.

Lastly, one verifies easily that $a(\lambda)\lesssim \lambda^{-1}$ and 
$R_{\delta}(\lambda)=S_\delta(\l)=T_\delta(\l)=0$. 
Since the injectivity radius of $\R\times S^{1}$ is smaller than $2\pi$, we  have $b_{\delta}(\lambda)\le 2\pi$.  
Thus,  \eqref{assumption} fails.

\begin{rem}\label{remark on counter example}
    Using a similar method, one shows that if we replace $V(x)$  by $c\cdot(\ln|x|)^\a$ with $\a\in(0,1),c>0$, then \eqref{failure} still holds. Moreover, a similar argument implies that there exists $\epsilon_0>0$ sufficiently small such that, after replacing $V(x)$  with $\epsilon_0 \ln|x|$, we have
\be\label{weyl law fails for r}
\limsup_{\l\to\infty}
\frac{N(\l)}{(2 \pi)^{-2} \omega_{2}\int_{\R\times S^1}(\l-V)_+\dvol_{\R\times S^1}}> 1,
\ee
and hence Weyl’s law fails in this case as well.
\end{rem}

\def\bH{\mathbb{H}}

\subsection{Examples on hyperbolic spaces}

Let $(\bH^{n}, g^{T\bH^{n}})$ denote $n$–dimensional hyperbolic space. Let $$V(x):=r(x)^\a,\quad0<\a<1$$ where $r(x)$ denotes the hyperbolic distance between $x$ and $0$.  
Let $\N(\l)$ be the eigenvalue counting function of $H:=\Delta+V$.

It is easily checked that $a(\lambda)\lesssim \lambda^{1-\a^{-1}}$, 
$R_{\delta}(\lambda)\equiv 1$, $S_\delta(\l)=T_\delta(\l)=0$ and $b_{\delta}(\lambda)\lesssim \delta^{2}$, 
so the condition \eqref{assumption} fails for this example.

For simplicity, we restrict attention to the case $\a=\half$ and $n=3$.
The general case $\a \in (0,1)$ and arbitrary $n$ can be treated in the same way.
 We will show that:

\begin{thm}\label{Weyl law fails}
    $\N(\l)=o\!\Big(\int_{\bH^3} (\l - V)_+^{3/2}\,\dvol_{\bH^3} \Big), \l\to\infty.$
\end{thm}
\begin{rem} 
  Before giving the proof, we present a naive argument. The heat kernel
$k(t,x,y)$ of $\Delta$ on three-dimensional hyperbolic space is given by
\be\label{heat kernel on hyperbolic space}
k(t,x,y)
= \frac{e^{-t}}{(4\pi t)^{3/2}}\,\frac{d(x,y)}{\sinh(d(x,y))}\,
e^{-\frac{d^2(x,y)}{4t}}.
\ee
Due to the presence of the factor $e^{-t}$ above, one may naively expect that
\[K_H(t,x,x)\sim (4\pi t)^{-\frac{3}{2}}e^{-t(V(x)+1)}, \quad t\to0.\]
Thus naively
\[
\N(\l)\sim (2\pi)^{-3}\omega_3
\int_{\bH^3}(\l-V-1)_+^{\frac{3}{2}}\,\dvol_{\bH^3},
\qquad \l\to\infty.
\]
Now if $V$ grows sufficiently slowly, 
$
\int_{\bH^3}(\l-V-1)_+^{\frac{3}{2}}\,\dvol_{\bH^3}
=
o\Big(
\int_{\bH^3}(\l-V)_+^{\frac{3}{2}}\,\dvol_{\bH^3}
\Big).
$
\end{rem}

Now we proceed to the rigorous proof.
First, by a direct computation analogous to \eqref{eq simple 10}, 
\begin{equation}\label{asymptotic of classical counting}
\int_{\bH^3} (\lambda - V)_{+}^{3/2}\,\dvol_{\bH^3}
\;\gtrsim\; \frac{\,e^{2\lambda^{2}}}{\lambda^{3/2}}.
\end{equation}
Recall that $\Omega_{\l}:=\{V(x)<\l\}$. It is straightforward to check that for any $\delta>0$,
\be\label{gap}
\mathrm{dist}\big(\Omega_{\l},\,\bH^3\setminus \Omega_{\l+\delta}\big)\geq 2\delta\l,
\ee
and 
\be\label{volume}
|\Omega_\l|\approx e^{2\l^2}.
\ee

Let 
\[
a=\tfrac12,\qquad 
Q_0:=\bH^3\setminus\Omega_{\l-a},
\qquad 
Q_1:=\Omega_{\l-a}.
\]
\newpage
Below is a direct corollary of \cite[Theorem~3.2]{wang1997global}:

\begin{prop}\label{neumann heat kernel estimates}
Let $k_i$, $i=0,1$, be the Neumann heat kernel of the Beltrami--Laplace
operator $\Delta$ on $Q_i$. Then if $\lambda>a$ is large enough, there exist $\lambda$-independent
constants $C_1,C_2>0$ such that for any $t>0$, $\delta>0$, and
$x,y\in Q_i$, one has
\[
0\le k_i(t,x,y)
\le (1+\delta)^{C_1}e^{\frac{10}{\delta}}
\big(\sinh(2\sqrt{t})-2\sqrt{t}\big)^{-1}
\exp\!\Big(-\frac{d^2(x,y)}{(4+\delta)t}\Big)
e^{C_2\delta t}.
\]
\end{prop}

\begin{proof}
Note that if $\l>a$ is sufficiently large, then
$\p Q_0=\p Q_1$, which are geodesic spheres of radius $(\l-a)^2$.
Hence $\p Q_0$ has bounded mean curvature when $\lambda$ is large enough.
By \cite[Theorem~3.2]{wang1997global}, the result follows.
\end{proof}

For any Schr\"odinger operator $L$ on a domain $\Omega$ with Lipschitz boundary, denote by
\[
\N(\l,L_{\Omega,N})
\]
the eigenvalue counting function of $L$ on $\Omega$ with Neumann boundary condition.

By the Rayleigh quotient argument, 
\be\label{eigencounting bounded by Neuman eigencounting}
\N(\l)\leq 
\N(\l, H_{Q_0,N})
+\N(\l,H_{Q_1,N}).
\ee
Using \Cref{neumann heat kernel estimates} with $\delta=1$ together with \Cref{lem:spectral-bound-from-heat-kernel} and \eqref{volume}, one sees,
\be\ba\label{upper bound for Q1}
\N(\l,H_{Q_1,N})
&\leq \N(\l, \Delta_{Q_1,N})
\leq e\int_{Q_1} k_1(\l^{-1},x,x)\,dx \\
&\lesssim \l^{3/2}|Q_1|
\lesssim\l^{3/2}e^{2(\l-\l^{-c})^2}
\lesssim\l^{3/2}e^{\,2\l^2-4\l^{1-c}}.
\ea\ee
Combining \eqref{asymptotic of classical counting} and \eqref{upper bound for Q1}, we conclude that
\be\label{Neuman eigencouting on Q1 is small o}
\N(\l, H_{Q_1,N})
=o\Big(\int_{\bH^3} (\l-V)_+^{3/2}\,\dvol_{\bH^3}\Big), \quad\l\to\infty.
\ee
Consider the subregions
\[
c=\frac{5}{8},\quad Q_0^1=\Omega_{\l-\l^{-c}}\setminus \Omega_{\l-a},\qquad
Q_0^2=\Omega_{\l+\l^{-c}}\setminus\Omega_{\l-\l^{-c}},
\qquad
Q_0^3=\bH^3\setminus\Omega_{\l+\l^{-c}}.
\]
For an elliptic operator $L$, let $e_L$ (resp. $K_L$) denote the pointwise eigenvalue counting
function (resp. heat kernel) introduced in \Cref{pointwise eigenvalue}. Proceeding as in \eqref{upper bound for Q1},  we obtain
\be\ba\label{integral on Q01 is small o}
\int_{Q_0^1} e_{H_{Q_0,N}}(\l,x,x)
&\leq e\int_{Q_0^1}K_{H_{Q_0,N}}(\l^{-1},x,x)
   \leq e\int_{Q_0^1}k_{0}(\l^{-1},x,x)  \\
&\lesssim\l^{3/2}e^{\,2\l^2 - 4\l^{\,1-c}}
 =o\!\left(\int_{\bH^3} (\l - V)_+^{3/2}\,\dvol_{\bH^3}\right), \l\to\infty.
\ea\ee
We next use Agmon estimates to control
$\int_{Q_0^3} e_{H_{Q_0,N}}(x,x)\,dx$.

\def\AG{{\mathrm{AG}}}
\def\dist{{\mathrm{dist}}}

\begin{lem}[Agmon estimate]\label{Agmon estimate}
Let $(M,g)$ be a complete manifold and let $f\in C(M)$.
Assume there exists a compact set $K\subset M$ such that
$\inf_{M\setminus K} f>0$.
Suppose that $0\le u\in W^{1,2}(M)$ and that
$(\Delta+f)u\le0$ on $M\setminus K$ in the weak sense, i.e.,
\[
\int_{M\setminus K} \nabla u\cdot\nabla v + fuv\,\le0,
\qquad
\forall\, 0\le v\in C_c^\infty(M\setminus K).
\]
Let $g_{\AG}:= f\cdot g$ be the Agmon metric on $M\setminus K$, and
$\dist_{\AG}$ be the associated distance.
Set $\rho(x):=\dist_{\AG}(x,\partial K)$ and
$
K_1 := K \cup \{x\in M\setminus K : \rho(x)\le 2\}.
$
Then for any $\beta\in(0,1)$,
\[
\int_{M\setminus K_1} f\,|u|^2 e^{2\b\rho}\,\mathrm{dvol}
\le
\frac{8e^{4\b}(1+\b^2)}{(1-\b^2)^2}
\int_{K_1\setminus K} f\,|u|^2\,\mathrm{dvol}.
\]
\end{lem}

\begin{proof}
The proof proceeds exactly as in \cite[Lemma~3.1]{DY2020cohomology} or
\cite[Theorem~1.5]{agmon2014lectures}.
\end{proof}

\begin{cor}\label{Cor Agmon}
Let $\psi$ be an eigenfunction of $H_{Q_0,N}$ with eigenvalue $\mu<\lambda$.
Then for any $\b\in(0,1)$, there exists a constant $C$ depending only on $\b$ such
that, for $\lambda$ sufficiently large,
\be\label{eq Cor Agmon}
\int_{\bH^3\setminus \Omega_{\lambda+\lambda^{-c}}} |\psi|^2
\le
C\,\lambda^{1+c}\, e^{-2\b\lambda^{1/16}}
\int_{\bH^3} |\psi|^2.
\ee
\end{cor}

\begin{proof}
By Kato's inequality, we have
$(\Delta+V-\mu)|\psi|\le0$ weakly on $\bH^3\setminus \Omega_\lambda$.
Let $\dist_\mu$ denote the Agmon distance associated with the Agmon metric
\[
(V-\mu)\, g^{T\bH^3}
\qquad\text{on }\bH^3\setminus \Omega_\lambda .
\]
On $\bH^3\setminus \Omega_{\lambda+\lambda^{-c}/2}$, we have
$
V-\mu \ge V-\lambda \ge \frac{\lambda^{-c}}{2}.
$
Using \eqref{gap}, this implies
\be\label{Agmon distance lambda}
\dist_\mu\!\big(\Omega_{\lambda+\lambda^{-c}/2},
\bH^3\setminus \Omega_{\lambda+\lambda^{-c}}\big)
\ge
\lambda^{-\frac{c}{2}}
\dist\!\big(\Omega_{\lambda+\lambda^{-c}/2},
\bH^3\setminus \Omega_{\lambda+\lambda^{-c}}\big)
\ge
\lambda^{1/16}.
\ee
Let $\rho_\mu(x):=\dist_\mu(x,\Omega_{\lambda+\lambda^{-c}/2})$.
By \eqref{Agmon distance lambda}, if $\lambda\ge 2^{16}$, then
\[
\Omega_{\lambda+\lambda^{-c}}
\supset
\Omega_{\lambda+\lambda^{-c}/2}
\cup
\{x:\rho_\mu(x)\le2\}.
\]
Applying \Cref{Agmon estimate}, we find that for any $b\in(0,1)$, there exists
a constant $C$ depending only on $\b$ such that
\[
\int_{\bH^3\setminus \Omega_{\lambda+\lambda^{-c}}}
(V-\mu)|\psi|^2 e^{2\b\rho_\mu}
\le
C
\int_{\Omega_{\lambda+\lambda^{-c}}} (V-\mu)|\psi|^2 .
\]
Finally, if $x\in\bH^3\setminus \Omega_{\lambda+\lambda^{-c}}$, then
$V-\mu\ge \lambda^{-c}$ and by \eqref{Agmon distance lambda}, $\rho_\mu(x)\ge \lambda^{1/16}$, while if
$x\in\Omega_{\lambda+\lambda^{-c}}$, we have $V-\mu\le 2\lambda$.
The estimate \eqref{eq Cor Agmon} follows.
\end{proof}

By \eqref{eq Cor Agmon}, we deduce that,
\be\label{integral on Q_0^3 is small o}
\int_{Q_0^3}e_{H_{Q_0,N}}(\l,x,x)
=\int_{\bH^3\setminus\Omega_{\l+\l^{-c}}}e_{H_{Q_0,N}}(\l,x,x)
\lesssim\,\N(\l,H_{Q_0,N})
\l^{1+c}\,e^{-\l^{1/16}}.
\ee
We now estimate
$
\int_{Q_0^2} e_{H_{Q_0,N}}(\lambda,x,x).
$
This is the part where we apply the idea from the naive argument. Let
\[
b=\frac{5}{12}.
\]
We will shift the operator by a constant. Let $L_2$ denote $H-\lambda+\lambda^{-b}$ on $Q_0$, equipped with Neumann boundary conditions. Then one checks immediately that
\[
e_{H_{Q_0,N}}(\l,x,x)=e_{L_2}(\l^{-b},x,x).
\]
Since on $Q_0$ we have $V-\l+\l^{-b}\geq -a$, it follows the maximal principle that
\[
K_{L_2}(t,x,y)\leq e^{at}\,k_0(t,x,y).
\]
Therefore, by \Cref{lem:spectral-bound-from-heat-kernel},
\be\label{bounds integral of eigencouting on Q02 using heat kernel}
\int_{Q_0^2}e_{H_{Q_0,N}}(\l,x,x)
\leq e\int_{Q_0^2} e^{a\l^b} k_0(\l^b,x,x)\,dx.
\ee
We now estimate $k_0(\l^b,x,x)$ for $x\in Q_0^2$.

\begin{lem}\label{estimate of k0 in Q02}
There exists a $\l$-independent constant $C>0$ such that,
\[
k_0(\l^{b},x,x)\leq C\,e^{-\l^b}\,\l^{3b/2},\quad x\in Q_0^2,\l\gg1.
\]
\end{lem}

\begin{proof}
Choose $\eta\in C^\infty(\bH^3\setminus\Omega_{\l-a})$ such that \def\tk{{\tilde{k}}}
\[
\eta|_{\Omega_{\l-a/2}\setminus\Omega_{\l-a}}\equiv 0,\qquad
\eta|_{\bH^3\setminus\Omega_{\l-a/4}}\equiv 1,\qquad
0\leq \eta\leq 1.
\]
Set $Q_0':=\Omega_{\l-a/4}\setminus \Omega_{\l-a/2}$ and let
\[
\tk_0(t,y,z)=\eta(z)\,k(t,y,z),\qquad y,z\in Q_0,
\]
where $k$ is the heat kernel on $\bH^3$ (see \eqref{heat kernel on hyperbolic space}).  
Then $(\p_t+\Delta)\tk_0$ is supported in $Q_0\times Q_0'$, with $\Delta$ acting in the second factor.

Note that for large $\l$, 
$
\mathrm{dist}(Q_0^2,Q_0')\geq \frac{a\l}{4}.
$
By \eqref{heat kernel on hyperbolic space}, for some $\l$-independent $c'>0$,
\be\label{estimate of remainder}
|(\p_t+\Delta)\tk_0(t,x,y)|
\leq e^{-t}e^{-c'd^{2-b}(x,y)},\quad x\in Q_0^2, y\in Q_0', t\in(0,\l^b).
\ee
Similarly, by \Cref{neumann heat kernel estimates}, we may as well assume that for the same $c'>0$,
\be\label{estimate of k0 when x and y are far away}
|k_0(t,x,y)|\leq e^{-t}e^{-c'd^{2-b}(x,y)}, \quad x\in Q_0^2, y\in Q_0', t\in(0,\l^b).
\ee
By Duhamel’s principle, \eqref{estimate of k0 when x and y are far away} and \eqref{estimate of remainder}, there exists $C>0$ such that for $x\in Q_0^2,$
\[\ba
|k_0(\l^b,x,x)-\tk_0(\l^b,x,x)|&\leq\int_0^{\l^b}\int_{Q_0'}|(\p_t+\Delta)\tk_0(s,x,y)||k_0(t-s,x,y)|dyds\\
&\leq  \l^be^{-\l^b} \int_{\bH^3}e^{-c'd^{2-b}(x,y)}dy \leq C \l^be^{-\l^b}
\ea\]
Note  also that by the construction of $\tk$ and \eqref{heat kernel on hyperbolic space}, we have $\tk_0(\l^b,x,x)=k(\l^b,x,x)\leq C\l^{3b/2}e^{-\l^b},$ the result follows.
\end{proof}

Using \Cref{estimate of k0 in Q02}, \eqref{volume} and the inequalities  $3/8=1-c<b=5/12$,
\be\ba\label{integral of heat kernels bounds on Q02}
\int_{Q_0^2} e^{a\l^b} k_0(\l^b,x,x)
&\lesssim\l^{3b/2} e^{-a\l^b}\,|Q_0^2|
   \lesssim\l^{3b/2} e^{2\l^2 + 2\l^{1-c} - a\l^b} \lesssim\l^{3b/2} e^{2\l^2 - \frac{a}{2}\l^b}.
\ea\ee
Hence, by \eqref{bounds integral of eigencouting on Q02 using heat kernel}, \eqref{integral of heat kernels bounds on Q02}, and \eqref{asymptotic of classical counting},
\be\label{integral on Q02 is small o}
\int_{Q_0^2} e_{L_1}(\l,x,x)
=
o\!\left(\int_{\bH^3} (\l - V)_+^{3/2}\,\dvol_{\bH^3}\right),\l\to\infty.
\ee
From \eqref{integral on Q_0^3 is small o}, for $\l\gg1$ we have
$
\N(\l,H_{Q_0,N})
\leq 
2\int_{Q_0^1\cup Q_0^2} e_{H_{Q_0,N}}(\l,x,x)\,dx.
$
Therefore, by \eqref{integral on Q01 is small o} and \eqref{integral on Q02 is small o},
\be\label{Neuman eigencouting on Q0 is small o}
\N(\l,H_{Q_0,N})
=
o\!\left(\int_{\bH^3}(\l - V)_+^{3/2}\,\dvol_{\bH^3}\right), \l\to\infty.
\ee
Finally, \Cref{Weyl law fails} follows from 
\eqref{eigencounting bounded by Neuman eigencounting}, 
\eqref{Neuman eigencouting on Q1 is small o}, 
and \eqref{Neuman eigencouting on Q0 is small o}.

\section{Examples Satisfying \eqref{assumption}}\label{examples satisfy assumptions}

\subsection{Slow-growing potentials}
In this subsection, we verify that the potentials
\[\ba
V(x)&=\ln\!\cdots\!\ln(|x|)\;,|x|\gg1\ \text{on }\mathbb{R}^{n},\qquad\\
V(x,\theta)&=\ln(|x|)^\a,\quad |x|\gg1,\ \a>1,\ (x,\theta)\in \mathbb{R}\times S^1,\\
V(x)&=r(x)^{\a}\;,\a>1\ \text{on hyperbolic space},
\ea\]
where $r(x)$ denotes the hyperbolic distance from $x$ to $0$, satisfy \eqref{assumption}.

For $V(x)=\ln(|x|)\;(|x|\geq1)$ on $\mathbb{R}^{n}$, one easily checks
\[
a(\lambda)\approx 1,\quad R_{\delta}(\lambda)=S_\delta(\l)=T_\delta(\l)=0,\quad b_{\delta}(\lambda)\approx e^{\delta^{2}\lambda/n}.
\]
Thus \eqref{assumption} holds. We can verify similarly that for $ V(x)=\ln(|x|)^\a,\ \a>0,\ |x|\gg1 $ or $ V(x)=\ln\!\cdots\!\ln(|x|),\ |x|\gg1 $, condition \eqref{assumption} holds.
\begin{rem}\label{compare R and R times S1}
    In contrast, the example in \Cref{R times s1}, on the product space $\mathbb{R}\times S^{1}$, fails to satisfy \eqref{assumption} because the upper bound on $b_{\delta}(\lambda)$ is constrained by injectivity-radius bounds and cannot admit exponential growth.
\end{rem} 

For
$
V(x,\theta)=\ln(|x|)^{\a}, |x|\gg1,\ \a>1,\ (x,\theta)\in \R\times S^1,
$
the validity of \eqref{assumption} follows from
\[
a(\lambda)\approx \lambda^{\frac{\a-1}{\a}},\qquad
R_{\delta}(\lambda)=S_\delta(\l)=T_\delta(\l)=0,\qquad
b_{\delta}(\lambda)\approx 2\pi.
\]
For
$
V(x)=r(x)^{\a}, \a>1,
$
on hyperbolic space, the validity of \eqref{assumption} is ensured by
\[
a(\lambda)\approx \lambda^{\frac{\a-1}{\a}},\qquad
R_{\delta}(\lambda)=1,\quad
S_\delta(\l)=T_\delta(\l)=0,\qquad
b_{\delta}(\lambda)\approx \delta^{2}.
\]

\subsection{Potentials satisfying doubling condition}\label{example in DY}

In the case of the potential satisfying the doubling condition \eqref{doubling condition0}, \Cref{main} recover the result of \cite{WeylDY}.
In \cite{WeylDY}, the discussion is limited to manifold with bounded geometry:
\begin{defn}[Bounded geometry]\label{defn bounded geometry}
    Let $(M,g)$ be a complete Riemannian manifold with metric $g$. We say $(M,g)$ has \emph{bounded geometry} if the following conditions hold:
    \begin{enumerate}[(1)]
        \item The injectivity radius $\tau$ of $(M,g)$ is uniformly bounded below by a positive constant.
        \item The norm of the curvature operator, as well as the norms of its first two
covariant derivatives, are uniformly bounded.
    \end{enumerate}
\end{defn}

In \cite{WeylDY}, the following function spaces are considered.

\begin{defn} \label{Def-a-reg}
Let $V \in L^\infty_{\mathrm{loc}}(M)$.
 For some $\b\in[0,\half]$, we say $V$ is $\b$-regular if there exists a decreasing continuous function $v: \R 
\mapsto (0,\infty)$ with $\lim_{t\to\infty}v(t)=0$, such that for any $x,y\in M$, whenever $d(x,y)<\tau$, we have 
		\begin{equation}\label{secondcond}
					|V(x)-V(y)|\leq d(x,y)^{2\b}\max\{|V(x)|^{1+\b},1\}v\big(V(x)\big).
				\end{equation}
This can be thought of as a quantified H\"older continuity condition for $V$.

We consider
\be\label{defn-Ra}
\mR_\b:=\left\{V\in L_{\mathrm{loc}}^\infty(M):\text{ $V$ satisfies \eqref{cond-V-2}, \eqref{secondcond}, and the doubling condition \eqref{doubling condition0} }\right\}.
\ee
\end{defn}
The function space $\mR_\b$ was also studied in \cite{rozenbljum1974asymptotics}, but the manifold is taken  to be $\R^n$.

If $V\in \mR_\b,\ \b\in[0,\half]$, one can verify easily that $a(\l)\approx \l$, $R_\delta(\l), S_\delta(\l),T_\delta(\l)\lesssim 1$, and $b_\delta(\l)\approx \l^{-\half}\left(\frac{\delta^2}{\nu(\l)}\right)^{\frac{1}{2\beta}}$ for some constant $c>1$, hence \eqref{assumption} holds for such $V$.

Another example of potentials considered in \cite{WeylDY} is:
For $\a \in [0, \tfrac{1}{2})$, let $\mS_\a$ be the class of functions satisfying the same conditions as $\mathcal{R}_\a$, except that \eqref{secondcond} is replaced by:
\begin{equation}\label{smooth-a}
V \in \mathrm{Lip}(M) \quad \text{and} \quad |\nabla V(x)| \leq C_V'' \max\{1, V(x)\}^{1+\a} \quad \text{a.e.},
\end{equation}
for some constant $C_V'' > 1$. Here $\mathrm{Lip}(M)$ denotes the space of Lipschitz functions on $M$.

The space ${\mS}_\a$, $\a \in [0, \tfrac{1}{2})$, was also considered by Tachizawa \cite[Theorem 4.3]{tachizawa1992eigenvalue} and Feigin \cite{feigin1976asymptotic}, where they limit their discussion to $\R^n$. It can be checked easily that for $V\in \mS_\a$,
$a(\l)\approx \l$, $R_\delta(\l), S_\delta(\l),T_\delta(\l)\lesssim 1$, and $b_\delta(\l)\approx \delta^2\l^{-\beta}$ for any $\beta\in(\a,\half)$ using arguments similar to those in \cite[$\S$ A.4]{WeylDY}, which imply \cite[Theorem 1.10]{WeylDY} in the case of $V\in \mS_\a$.

We can also consider the following space.
Let $\widetilde{\mathcal{R}}_0$ be the space of functions satisfying the same conditions as $\mathcal{R}_\a$, except that \eqref{secondcond} is replaced by the following: there exists an increasing function $\eta \in C([0, \tau))$ with $\eta(0) = 0$ such that for almost every $d(x,y)<\tau$,
\begin{equation}\label{smooth-0}
|V(x) - V(y)| \leq \eta\big(d(x, y)\big) \max\{1, |V(x)|\}.
\end{equation}
The space $\widetilde{\mR}_0$ was considered by Fleckinger \cite{fleckinger1981estimate}, and the discussion in that paper is still limited to $\mathbb{R}^n$. We then have $a(\lambda)\approx \lambda$, $R_\delta(\l), S_\delta(\l),T_\delta(\l)\lesssim 1$, and $b_{\delta}(\lambda)\approx 1$, thus \eqref{assumption} holds. Consequently, our \Cref{main} implies \cite[Theorem 1.10]{WeylDY} in this case.

It is also possible to extend the analysis to potentials with weaker regularity, 
following the approach in \cite{WeylDY} where an integral oscillation condition (such as the function space $\mathcal{O}_\b$) is used, but the argument in this paper would not be as concise as showed here.
\def\newexample{1}
\if\newexample0
\subsection{More examples}

We will introduce more examples of potentials satisfying \eqref{assumption} that is closely related to examples in \cref{example in DY}.

Let $a:[1,\infty)\to(0,\infty)$ be a monotone function such that $a(\l)\lesssim \l$. We will introduce a class of function that is closely related to $\mR_\b$ in \cref{example in DY}.

\begin{defn} \label{Def-a-reg}
Let $V \in L^\infty_{\mathrm{loc}}(M)$.
 For some $\b\in[0,\half]$, we say $V$ is $(a,\b)$-regular if there exists a decreasing continuous function $v: \R 
\mapsto (0,\infty)$ with $\lim_{t\to\infty}v(t)=0$ such that
\begin{itemize}

   \item If $\lim_{\l\to\infty}a(\l)=\infty$, for any $x,y\in M$, whenever $d(x,y)<\tau$ such that $V(x)\geq1$, we have 
		\begin{equation}\label{secondcond}
					|V(x)-V(y)|\leq d(x,y)^{2\b}a^{1+\b}\big(V(x)\big)v\big(V(x)\big).
				\end{equation}
         \item If $\lim_{\l\to\infty}a(\l)\neq\infty$, for any $x,y\in M$ such that $V(x)\geq 1$, we have 
		\begin{equation}\label{secondcond}
					|V(x)-V(y)|\leq d(x,y)^{2\b}a^{1+\b}\big(V(x)\big)v\big(V(x)\big).
				\end{equation}
                Moreover, when $V(x)$ is large enough, the injectivity radius at $x$ is bounded below by $a^{-\half}\big(V(x)\big) v^{-\frac{1}{2\beta}}\big(V(x)\big).$
\end{itemize}

We consider
\be\label{defn-Ra}
\mR_\b(a):=\left\{V\in L_{\mathrm{loc}}^\infty(M):\text{ $V$ satisfies \eqref{cond-V-2}, \eqref{secondcond}, and \eqref{new defn of a} }\right\}.
\ee

Then one can check that for $V\in \mR_\b(a)$, $a(\l)=a(\l)$, $b_\delta(\l)\approx\delta^2a^{-\half}\big(V(x)\big) v^{-\frac{1}{2\beta}}\big(V(x)\big)$ 
\end{defn}
\fi

\appendix
\section{Proof of \Cref{tesslation} }\label{appendix}

 Throughout the proof, all constants depend only on $n,R$ and $r_0$, unless stated explicitly. 

  Item (1) and item (2) follows easily from our construction and volume comparison. 

 Next, we prove item (3). Fix $j$ and $x_* \in \partial \tQ_j$. Set $f_{lj}(x) := d(x, x_l) - d(x, x_j),l\neq j$ and $f_{jj}(x):= d(x, x_j) - r_0$. Then $\partial \tQ_j$ can be expressed as the intersection of finitely many level sets of the form $\{f_{lj} = 0\}$. Set
\[
J(x_*) := \bigl\{ l \in \{1, 2, \dots, J\} : f_{lj}(x_*) = 0 \bigr\}.
\]
Then $J(x_*)$ contains at most $N$ elements, where $N$ is the constant appearing
in \Cref{Vitali covering0}.

One suffices to deal with $l \in J(x_*)\setminus\{j\}$. For such $l$,
\begin{equation}\label{tess eq1}
\nabla f_{lj}(x_*) = u_l - u_j,
\end{equation}
where $u_l, u_j \in T_{x_*}X$ satisfy $\exp_{x_*}(u_l) = x_l$ and $\exp_{x_*}(u_j) = x_j$.

Since $5^{-1}r_0\leq d(x_*, x_j) = d(x_*, x_l) \leq r_0$ and $5^{-1}r_0\leq d(x_j, x_l)$, it follows from Toponogov comparison that the angle between $u_j$ and $u_l$ is bigger than $\theta_0$  for some constant $\theta_0\in(0,\pi].$

Let $\nu = \dfrac{u_j}{\sqrt{g(u_j, u_j)}}$. Note that $u_j$ and $u_l$ have the same length, and we obtain
\begin{equation}\label{nonvanishing derivative}
\left|g\bigl( \nabla f_{l,j}(x_*), \nu \bigr)\right|=\big(1-\cos(\theta_0)\big)|u_l| \geq 5^{-1}r_0\big(1-\cos(\theta_0)\big)>0.
\end{equation}
We identify $T_{x_*}X$ with $\mathbb{R}^n$ so that $\nu$ corresponds to $e_n:=(0, \dots, 0, 1)$. Let $\varphi$ be the inverse of the exponential map at $x_*$.  By \cite[Corollary~6.6.1]{jost2005riemannian}, the chart $\varphi$ satisfies
\eqref{Lip varphi0} with some constant $L_1$, for all radii $r_1<r_0$.

By \eqref{nonvanishing derivative} and the implicit function theorem, there exists a neighborhood $U$ of $x_*$ and a smooth function $\psi_l \in C^{\infty}(\mathbb{R}^{n-1})$ such that the graph of $\psi_l$ in $\varphi(U)$ gives $\{ f_{lj} \circ \varphi^{-1} = 0 \}$.

To obtain the uniform bounds $ L_2, r_1$, note that by the Taylor expansion
\begin{equation}\label{Taylor gradient}
\Bigl| \p_{y_n}(f_{lj} \circ \varphi^{-1})(y) \Bigr|
= \Bigl| \p_{y_n}(f_{lj} \circ \varphi^{-1})(0) + R_{lj}(y) \Bigr|
\geq 5^{-1}r_0\bigl(1 - \cos(\theta_0)\bigr) - |R_{lj}(y)|.
\end{equation}
By the Hessian comparison theorem \cite[Theorem~6.6.1]{jost2005riemannian},
we have if $5^{-1}r_0<d(x,x_j),d(x,x_l)< 2r_0$, then for some constant $C>0$,
\be\label{Hess flj}
|\Hess\, f_{l,j}(x)|
\le |\Hess\, d(x,x_l)|+|\Hess\, d(x,x_j)|
\le C.
\ee
 Thus, Taylor's theorem yields
$
|R_{lj}(y)|\le C'|y|
$
for some $C'$. Hence, we can see that
if $|y|<r_1$ for some sufficiently small $r_1$, the right-hand side of \eqref{Taylor gradient} is large than
$
10^{-1}r_0\bigl(1-\cos(\theta_0)\bigr).
$

By the implicit function theorem, the Lipchitz norm of $\psi_l$ is bounded by a constant $L_2 > 0$. Set $\psi(y') :=\min_{l}\{ \psi_{l}(y') \}.$ Then $\psi$ is a Lipchitz function whose graph gives $\partial \tQ_j \cap \varphi(U)$, and the Lipchitz norm of $\psi$ is also bounded by $L_2$. This completes the proof of item (3). 

Next, we address item (4). Set $\Sigma_{jl} := B_{r_0}(x_j) \cap \{ f_{jl} = 0 \}$. By \eqref{Hess flj}, we have
$
|{\rm Hess} f_{lj}| \Big|_{\Sigma_{jl}} \leq C.
$
Thus, the sectional curvature of $\Sigma_{jl}$ is bounded by a constant; consequently, by volume comparison, the area of $\Sigma_{jl}$ is bounded  by some constant times $r_0^{n-1}$. Similarly, there exists $\epsilon_0 > 0$ such that whenever $\epsilon < \epsilon_0$, the volume of the $\epsilon$-neighborhood  of $\Sigma_{jl}$ is bounded by $C' \epsilon r_0^{n-1}$ for some constant $C'$. Since $\tQ_j^\epsilon$ is contained in at most $N$ such $\epsilon$-neighborhoods, we obtain item (4).

Lastly, item (5) follows from \Cref{Varo}, \Cref{Cone implies SObolev} and item (3).
		\bibliography{lib}

@article{bonthonneau2015weyl,
  title={Weyl laws for manifolds with hyperbolic cusps},
  author={Bonthonneau, Y.},
  journal={arXiv preprint arXiv:1512.05794},
}

@article{boulkhemair2007uniform,
  title={On the uniform {P}oincar{\'e} inequality},
  author={Boulkhemair, A. and Chakib, A.},
  journal={Communications in Partial Differential Equations},
  volume={32},
  number={9},
  pages={1439--1447},
  year={2007},
  publisher={Taylor \& Francis}
}

@article{sturm1996analysis,
  title={Analysis on local {D}irichlet spaces. {III}. The parabolic {H}arnack inequality},
  author={Sturm, K.-T.},
  journal={J. Math. Pures Appl.},
  year={1996},
page={273–297},
volume={75},
number={3}
}

@article{sturm1995analysis,
  title={Analysis on local {D}irichlet spaces. {II}. Upper {G}aussian estimates for the fundamental solutions of parabolic equations},
  author={Sturm, K.-T.},
  year={1995},
page={275-312},
journal={Osaka J. Math.},
volume={32},
number={2}
}

@book{fanghua2003geometric,
  title={Geometric measure theory. {A}n introduction},
  author={Lin, F. and Yang, X.},
  year={2003},
  publisher={International Press}
}

@article{chitour2024weyl,
  title={Weyl's law for singular {R}iemannian manifolds},
  author={Chitour, Y. and Prandi, D. and Rizzi, L.},
  journal={Journal de Math{\'e}matiques Pures et Appliqu{\'e}es},
  volume={181},
  pages={113--151},
  year={2024},
  publisher={Elsevier}
}

@article{fleckinger1981estimate,
  title={Estimate of the number of eigenvalues for an operator of {S}chr{\"o}dinger type},
  author={Fleckinger, J.},
  journal={Proceedings of the Royal Society of Edinburgh},
  volume={89},
  number={3-4},
  pages={355--361},
  year={1981},
  publisher={Royal Society of Edinburgh Scotland Foundation}
}

@article{feigin1976asymptotic,
  title={Asymptotic distribution of eigenvalues for hypoelliptic systems in {$\mathbb{R}^n$}},
  author={Feigin, V. I.},
  journal={Matematicheskii Sbornik},
  volume={141},
  number={4},
  pages={594--614},
  year={1976},
  publisher={Russian Academy of Sciences, Steklov Mathematical Institute of Russian~}
}

@article{MP-WEYL,
  title={Eigenvalue asymptotics for a class of md-elliptic $\psi$do’s on manifolds with cylindrical exits},
  author={ Maniccia, L. and Panarese, P.},
  journal={Annali di Matematica Pura ed Applicata},
  volume={181},
  number={3},
  pages={1618--1891},
  year={2002},
}

@article{moroianu2008weyl,
  title={Weyl laws on open manifolds},
  author={Moroianu, S.},
  journal={Mathematische Annalen},
  volume={340},
  number={1},
  pages={1--21},
  year={2008},
  publisher={Springer}
}

@article{levendorskiui1996spectral,
  title={Spectral asymptotics with a remainder estimate for {S}chr{\"o}dinger operators with slowly growing potentials},
  author={Levendorski{\u{i}}, S.Z.},
  journal={Proceedings of the Royal Society of Edinburgh},
  volume={126},
  number={4},
  pages={829--836},
  year={1996},
  publisher={Royal Society of Edinburgh Scotland Foundation}
}

@article{coriasco2021weyl,
  title={Weyl law on asymptotically {E}uclidean manifolds},
  author={Coriasco, S. and Doll, M.},
  journal={Annales Henri Poincar{\'e}},
  volume={22},
  pages={447--486},
  year={2021},
}

@article{rofe1970conditions,
  title={Conditions for the self-adjointness of the {S}chr\"odinger operator},
  author={Rofe-Beketov, F. S.},
  journal={Mathematical notes of the Academy of Sciences of the USSR},
  volume={8},
  number={6},
  pages={888--894},
  year={1970},
  publisher={Springer}
}

@article{oleinik1994connection,
  title={On a connection between classical and quantum-mechanical completeness of the potential at infinity on a complete {R}iemannian manifold},
  author={Oleinik, I.},
  journal={Math. Notes.},
  volume={55},
  number={4},
  year={1994},
pages={380-386}
}

@article{canzani2013analysis,
  title={Analysis on manifolds via the {L}aplacian},
  author={Canzani, Y.},
  journal={Lecture Notes available at: http://www. math. harvard. edu/canzani/docs/Laplacian. pdf},
  year={2013}
}

@article{tachizawa1992eigenvalue,
  title={Eigenvalue asymptotics of {S}chr{\"o}dinger operators with only discrete spectrum},
  author={Tachizawa, K.},
  journal={Publications of the Research Institute for Mathematical Sciences},
  volume={28},
  number={6},
  pages={943--981},
  year={1992},
  publisher={Research Institute forMathematical Sciences}
}

@article{rozenbljum1974asymptotics,
  title={{A}symptotics of the eigenvalues of the {S}chr{\"o}dinger operator},
  author={{R}ozenbljum, {G. V.}},
  journal={{M}athematics of the {USSR-S}bornik},
  volume={22},
  number={3},
  pages={349},
  year={1974},
  publisher={{IOP P}ublishing}
}

@book{agmon2014lectures,
  title={Lectures on exponential decay of solutions of second-order elliptic equations: {B}ounds on eigenfunctions of {N}-body {Schr{\"o}dinger} operations.(MN-29)},
  author={Agmon, S.},
  year={2014},
  publisher={Princeton University Press}
}

@article{DY2020cohomology,
  title={Witten deformation for noncompact manifolds with bounded geometry},
  author={Dai, X. and Yan, J.},
  journal={Journal of the Institute of Mathematics of Jussieu},
  volume={22},
  number={2},
  pages={643--680},
  year={2023},
  publisher={Cambridge University Press}
}

@article{braverman2025semi,
  title={The semi-classical {W}eyl law on complete manifolds},
  author={Braverman, M.},
  journal={arXiv preprint arXiv:2505.12157},
}

@article{WeylDY,
  title={Weyl Law for {S}chr\"odinger Operators on Noncompact Manifolds, Heat Kernel, and {Karamata-Hardy-Littlewood} Theorem},
  author={Dai, X. and Yan, J.},
  journal={arXiv preprint arXiv: 2504.15551},
}

@book{grigoryan2009heat,
  title={Heat kernel and analysis on manifolds},
  author={Grigoryan, A.},
  volume={47},
  year={2009},
  publisher={American Mathematical Soc.}
}

@article{grigor1994heat,
  title={Heat kernel upper bounds on a complete non-compact manifold},
  author={Grigor'yan, A.},
  journal={Revista Matem{\'a}tica Iberoamericana},
  volume={10},
  number={2},
  pages={395--452},
  year={1994},
  publisher={Madrid: Revista Matematica Iberoamericana}
}

@article{varopoulos1985hardy,
  title={Hardy-{L}ittlewood theory for semigroups},
  author={Varopoulos, N. Th.},
  journal={Journal of functional analysis},
  volume={63},
  number={2},
  pages={240--260},
  year={1985},
  publisher={Academic Press}
}

@article{grigor1997gaussian,
  title={Gaussian upper bounds for the heat kernel on arbitrary manifolds},
  author={Grigor’yan, A.},
  journal={J. Diff. Geom},
  volume={45},
  number={1},
  pages={33--52},
  year={1997}
}

@book{adams2003sobolev,
  title={Sobolev spaces},
  author={Adams, R. A.  and Fournier, J. },
  volume={140},
  year={2003},
  publisher={Elsevier}
}

@book{jost2005riemannian,
  title={Riemannian geometry and geometric analysis},
  author={Jost, J.},
  year={2005},
  publisher={Springer}
}

@article{wang1997global,
  title={Global heat kernel estimates},
  author={Wang, J.},
  journal={Pacific Journal of Mathematics},
  volume={178},
  number={2},
  pages={377--398},
  year={1997},
  publisher={Mathematical Sciences Publishers}
}

@article{weyl1911asymptotische,
  title={{\"U}ber die asymptotische Verteilung der Eigenwerte},
  author={Weyl, H.},
  journal={Nachrichten von der Gesellschaft der Wissenschaften zu G{\"o}ttingen, Mathematisch-Physikalische Klasse},
  volume={1911},
  pages={110--117},
  year={1911}
}

@article{inahama2004eigenvalues,
  title={Eigenvalue asymptotics for the {S}chr{\"o}dinger operators on the real and the complex hyperbolic spaces},
  author={Inahama, Y. and Shirai, S.-I.},
  journal={Journal de Math{\'e}matiques Pures et Appliqu{\'e}es},
  volume={83},
  number={5},
  pages={589--627},
  year={2004},
  publisher={Elsevier}
}

@article{inahama2004eigenvalue,
  title={Eigenvalue asymptotics for the {S}chr{\"o}dinger operators on the hyperbolic plane},
  author={Inahama, Y. and Shirai, S.-I.},
  journal={Journal of Functional Analysis},
  volume={211},
  number={2},
  pages={424--456},
  year={2004},
  publisher={Elsevier}
}

@article{de1950asymptotic,
  title={On the asymptotic distribution of eigenvalues},
  author={De Wet, J. S. and Mandl, F.},
  journal={Proceedings of the Royal Society of London. },
  volume={200},
  number={1063},
  pages={572--580},
  year={1950},
  publisher={The Royal Society London}
}
		\bibliographystyle{plain}
	\end{document}